\documentclass[10pt]{amsart}
\usepackage{fullpage}
\usepackage{amssymb,amsmath,amsthm,amscd,mathrsfs,graphicx, color}
\usepackage[cmtip,all,matrix,arrow,tips,curve]{xy}
\usepackage[final]{pdfpages}

\numberwithin{equation}{section}
\theoremstyle{plain}
\newtheorem{theo}[equation]{Theorem}

\newtheorem*{claim*}{Claim}
\newtheorem{lem}[equation]{Lemma}

\newtheorem{prop}[equation]{Proposition}
\newtheorem{coro}[equation]{Corollary}
\newtheorem{conjecture}[equation]{Conjecture}

\theoremstyle{remark}
\newtheorem{rema}[equation]{Remark}

\theoremstyle{definition}
\newtheorem{definition}[equation]{Definition}

\def\Sym{\operatorname{Sym}}

\newcommand{\bP}{\mathbb{P}}

\newcommand{\bR}{\mathbb{R}}
\newcommand{\bQ}{\mathbb{Q}}

\newcommand{\bC}{\mathbb{C}}

\newcommand{\calC}{\mathcal{C}}

\newcommand{\calX}{\mathcal{X}}

\newcommand{\Gr}{\mathrm{Gr}}

\newcommand{\wt}{\widetilde}

\newcommand{\mc}{\mathcal}

\DeclareMathOperator{\Def}{Def}

\DeclareMathOperator{\Sk}{Sk}

\DeclareMathOperator{\SU}{SU}
\DeclareMathOperator{\spe}{sp}
\DeclareMathOperator{\Sp}{Sp}
\DeclareMathOperator{\nilp}{nilp}
\newcommand{\simq}[0]{\sim_{\bQ}}
\newcommand{\map}[0]{\dasharrow}

\newcommand{\be}{\begin{equation}}
\newcommand{\ee}{\end{equation}}

\begin{document}

 \title[Degenerations of hyper-K\"ahlers]{Remarks on degenerations of hyper-K\"ahler manifolds}

 \author[J. Koll\'ar]{J\'anos Koll\'ar}
 \address{Princeton University,  Princeton, NJ 08544, USA}
\email{kollar@math.princeton.edu}

 \author[R. Laza]{Radu Laza}
\address{Stony Brook University,  Stony Brook, NY 11794, USA}
\email{radu.laza@stonybrook.edu}

 \author[G. Sacc\`a]{Giulia Sacc\`a}
\address{Stony Brook University,  Stony Brook, NY 11794, USA}
\email{giulia.sacca@stonybrook.edu}

\author[C. Voisin]{Claire Voisin}
\address{Coll\`ege de France, 3 rue d'Ulm, 75005 Paris, France}
\email{claire.voisin@imj-prg.fr}

\begin{abstract}
Using the Minimal Model Program, any degeneration of $K$-trivial varieties can be arranged to be in a Kulikov type form, i.e.\  with trivial relative canonical divisor and mild singularities.  In the hyper-K\"ahler setting, we can then deduce a finiteness   statement
for monodromy acting on $H^2$, once one knows that  one component of the central fiber is not uniruled.  Independently of this, using deep results from the geometry of hyper-K\"ahler manifolds, we prove that a finite monodromy projective degeneration of hyper-K\"ahler manifolds has a smooth filling (after base change and birational modifications). As a consequence of these two results, we  prove a generalization of Huybrechts'
theorem about birational versus deformation equivalence, allowing singular
central fibers. As an application, we give simple proofs for the deformation type of certain geometric constructions of hyper-K\"ahler manifolds (e.g. Debarre--Voisin \cite{debarrevoisin} or Laza--Sacc\`a--Voisin \cite{lsv}). In a slightly different direction, we establish some basic properties (dimension and rational homology type) for the dual complex of a Kulikov type degeneration of hyper-K\"ahler manifolds.
\end{abstract}
\date{\today}
\maketitle

\bibliographystyle{amsalpha}

\section*{Introduction}
The starting point of this note was the study of deformation types of hyper-K\"ahler manifolds. By hyper-K\"ahler manifold we will mean a hyper-K\"ahler manifold which is also compact and irreducible.
Recall the  following fundamental result due to Huybrechts:
\begin{theo} \label{theohuy} (Huybrechts \cite{huybrechts}) Let $X$ and $X'$ be two birationally equivalent projective
hyper-K\"ahler manifolds. Then $X$ and $X'$ are deformation equivalent. (More precisely, $X$ and $X'$ have  arbitrarily small deformations that are  isomorphic to each other.)
\end{theo}

Equivalently, if $\mathcal{X}\rightarrow \Delta$ is a family of smooth hyper-K\"ahler manifolds
with central fiber $X_0$ bimeromorphic to a hyper-K\"ahler manifold
$X'_0$, then the fibers $X_t$ are deformation equivalent to $X'_0$. One of the results of our paper is a version of the last statement allowing a singular
fiber $X_0$, at least if the fibers are projective. Specifically, the following holds:
\begin{theo}\label{theo1}
Let $\mathcal{X}\rightarrow \Delta$ be a projective morphism with $X_t$ smooth
hyper-K\"ahler for $t\not=0$. Assume that at least one irreducible reduced  (that is, multiplicity $1$)
component of the central fiber
is birational to a smooth hyper-K\"ahler manifold $X'_0$.
Then the smooth fibers $X_t$ are deformation equivalent to  ${X}'_0$.
\end{theo}

 Theorem \ref{theo1} is very useful in practice, as there are many examples of degenerating families of hyper-K\"ahler manifolds with the central fiber birational to a hyper-K\"ahler manifold. In fact, as explained below, this theorem significantly simplifies arguments given in  \cite{lsv}, \cite{debarrevoisin}, and other papers (see Section \ref{secexamples}) about the deformation type of certain constructions
  of explicit projective models of hyper-K\"ahler manifolds.

Theorem \ref{theo1} is  a generalization of Huybrechts' theorem but the latter is in fact very much used in the proof.  Namely, the  proof of  Theorem \ref{theo1} follows from Huybrechts' theorem using the following new result.

\begin{theo}\label{theo2} Let $\mathcal{X}\rightarrow \Delta$ be a projective morphism with general fiber $X_t$ a smooth
hyper-K\"ahler manifold. Assume that at least one irreducible component of the central fiber $X_0$ is not uniruled.
Then after a finite base change
$S\rightarrow \Delta$, the family
$\mathcal{X}_S:=\mathcal{X}\times_{\Delta}S\rightarrow S$ is bimeromorphic
over $S$ to a family $\pi':\mathcal{X}'\rightarrow S$ that
is  smooth and proper over $S$ with projective hyper-K\"ahler fibers.
 \end{theo}
\begin{rema}  {\rm The assumption on the central fiber is satisfied if the desingularization of  one irreducible component $V$ of $X_0$ has a generically nondegenerate holomorphic $2$-form
and this is the main situation where we will apply the theorem. Note that with this stronger assumption,   Theorem \ref{theo2}
  was previously announced by Todorov \cite{todorov}.}
  \end{rema}

 \begin{rema} {\rm In the assumptions of Theorem \ref{theo2}, we did not ask that the considered
irreducible component $V$  be reduced. This is because after base change and normalization, we can remove multiplicities, still having a component satisfying the main assumption, but now reduced. In this process, the considered
component $V$, when it has multiplicity $>1$, is replaced by a generically finite cover of $V$, hence it is not in general birational to $V$. This is why
we need the multiplicity $1$ assumption in Theorem \ref{theo1}, whose statement actually  involves the birational model of $X_0$.}
\end{rema}

 A first important step in the proof of Theorem \ref{theo2}  is
the following result:
\begin{theo}\label{theorr} Let $\mathcal{X}\rightarrow \Delta$ be a projective morphism with general fiber $X_t$ a smooth
hyper-K\"ahler manifold.  Assume that at least one irreducible component of the central fiber $X_0$ is not uniruled.
Then, the monodromy action on the
degree $2$ cohomology of the smooth fiber $X_t$ is finite.
\end{theo}

Once one has finiteness of the monodromy acting on $H^2$, Theorem \ref{theo2} is  a consequence
of the following  variant of Theorem
\ref{theo2} whose proof  uses the surjectivity of the period map proved by Huybrechts and Verbitsky's Torelli theorem (see \cite{verbi}, and also \cite{huybrechtsbourbaki}).
\begin{theo}\label{theo3} Let $\mathcal{X}\rightarrow \Delta$ be a projective morphism with general fiber $X_t$ a smooth
hyper-K\"ahler manifold. Assume the monodromy acting on
$H^2(X_t,\mathbb{Q}) $ is finite.
Then after a finite base change
$S\rightarrow \Delta$, the family
$\mathcal{X}_S:=\mathcal{X}\times_{\Delta}S\rightarrow S$ is bimeromorphic
over $S$ to a family $\pi':\mathcal{X}'\rightarrow S$ which
is  smooth proper over $S$ with projective hyper-K\"ahler fibers.
 \end{theo}
\begin{rema}\label{remfinitemoncy} Let us emphasize that this is a result specific to hyper-K\"ahler manifolds. There are examples of  families of Calabi-Yau varieties for which the monodromy is finite, but  not admitting a smooth filling after base change. The first example is due to Friedman \cite{friedman} who noticed that  a generic degeneration to a quintic threefold with an $A_2$ singularity has finite order monodromy.  Wang \cite[\S4]{wang} then checked that there is no smooth projective filling.
Another example, this time for Calabi-Yau fourfolds, is that of a Lefschetz $1$-nodal  degeneration of a sextic hypersurface in $\mathbb{P}^5$ which is treated in \cite{voisinlefdeg}. For this example, Morgan
\cite{morgan} shows that
the monodromy is finite in the group of isotopy classes of diffeomorphisms
of the smooth fibers $X_t$, so that after finite base change, the family admits
a $\mathcal{C}^\infty$ filling. It is proved in \cite{voisinlefdeg} that
 the base-changed family doesn't
admit a filling with a smooth Moishezon fiber for any base change.
\end{rema}

Theorem \ref{theo3} tells us that under the same assumptions on $\mathcal{X}\rightarrow \Delta$,
there is, after base change, a  family $\mathcal{X}'\rightarrow \Delta$ birationally equivalent to
$\mathcal{X}$ over $\Delta$, with smooth central fiber. The monodromy
action on the whole cohomology of the fiber
$X'_t$ is thus finite. With a little more work, we will prove
in Section \ref{secmain} that the monodromy action on the whole cohomology of the original fiber
$X_t$ is also finite (see Corollary \ref{corofinitemonoevrywhere}). (Note that $X_t$ and $X'_t$ are isomorphic in codimension $1$, but they typically differ in higher codimensions.)

\medskip

Theorem \ref{theorr} rests on the application of the minimal model program (MMP)   (see Section \ref{secmmp})
 to understanding the degenerations of K-trivial varieties (such as Calabi-Yau or hyper-K\"ahler manifolds). For a long time it was understood that the MMP plays a central role in this enterprise. Namely, the model result here is the Kulikov--Persson--Pinkham (KPP) Theorem which says that a $1$-parameter degeneration of $K3$ surfaces can be arranged to be a semistable family satisfying the additional condition that the relative canonical class is trivial.  As an application of this result, one obtains  control of the monodromy for the degenerations of $K3$ surfaces in terms of the central fiber and then a properness result for the period map. In higher dimensions, the analogue of the KPP  theorem is that any $1$-parameter degeneration of K-trivial varieties can be modified such that all the fibers have mild singularities and that the relative canonical class is trivial (this is nothing but a relative minimal model). More precisely,  a  higher dimensional analogue of the KPP theorem is given by Fujino \cite{fujino-ss} and Lai \cite{lai-2009} (building on \cite{BCHM}). We state a refinement of Fujino's result
 in Theorem \ref{thmkulikovmodel}, which provides some additional control on the behavior of the central fiber under the semistable reduction (needed to achieve mild singularities), followed by the minimal model program (need to achieve $K$-triviality).

 To complete the proof of Theorem \ref{theorr}, we use the fact that the singularities occurring in the MMP are mild from a cohomological point of view. This follows by combining the results of Koll\'ar--Kov\'acs \cite{kk} and Steenbrink \cite{steenbrink}, which give a vast generalization and deeper understanding of the results of Shah \cite{shahinsignificant,shah2} on degenerations of $K3$ surfaces. These arguments apply  to degenerations of any $K$-trivial varieties, but since the cohomologically mild condition refers only to the holomorphic part of the cohomology (i.e. the $H^{k,0}$ pieces of the Hodge Structure), controlling the monodromy in terms of the central fiber is possible only for $H^1$ and $H^2$ (see Theorems \ref{theomon1} and \ref{theorr} for the case of hyper-K\"ahler fibers). Since the degree $2$ cohomology controls the geometry of hyper-K\"ahler manifolds, we obtain in Section \ref{secmain} the much stronger result (that can not follow from general MMP) that certain degenerations of hyper-K\"ahler manifolds have smooth fillings.

As explained above, the smooth filling of finite monodromy degenerations (Theorem \ref{theo3}) is a result specific to hyper-K\"ahler manifolds. The proof given in Section \ref{secmain} depends on deep properties (Torelli and surjectivity) of the period map.
In Section \ref{secsym}, we give a completely different  proof of Theorems \ref{theo2} and \ref{theo3}, which again depends on specific results in the geometry of hyper-K\"ahler manifolds. Specifically, starting with a degeneration of hyper-K\"ahler manifolds $\calX/\Delta$ with a component of the central fiber not uniruled, by applying the MMP results of Section \ref{secmmp} and ideas similar to those in Section \ref{secmono}, we conclude that the central fiber $X_0$ can be assumed to have symplectic singularities in the sense of Beauville \cite{Beauville}. The rigidity results of Namikawa \cite{namikawa1,namikawa2} then allow us to conclude that the degeneration can be modified to  give a smooth family.

 Theorem \ref{theo1}, and even its weaker version Theorem \ref{theo3}
are very useful in practice and we will devote Section
\ref{secexamples} to describing a number of geometric examples. The most important one, which was the original motivation for this paper, is the case of the intermediate Jacobian fibration  associated to a cubic fourfold.
Specifically, we recall that in \cite{lsv} we have given, starting with a cubic fourfold $W$, a construction of a 10-dimensional hyper-K\"ahler manifold $X$ compactifying the intermediate Jacobian
fibration associated to the family of smooth hyperplane sections of $W$.  We then proved, via delicate geometric arguments (\cite[Section 6]{lsv}),  that when the cubic fourfold is Pfaffian,  the so-constructed hyper-K\"ahler manifold specializes well and is birational  to an O'Grady's $10$-dimensional exceptional hyper-K\"ahler manifold (whose deformation class is referred to as OG10 in this paper). By Huybrechts' theorem \cite[Theorem 4.6]{huybrechts} (Theorem \ref{theohuy} above),
we concluded that our compactified intermediate Jacobian fibrations
are deformation equivalent to OG10. While our arguments in \cite{lsv} establish the desired result, they are somewhat convoluted and obscure, as Pfaffian geometry is beautiful but
sophisticated. As observed by O'Grady and Rapagnetta (\cite{ogradyrapagnetta}) even before we started working on \cite{lsv}, another degeneration linking  in a more direct way the intermediate Jacobian
 fibration to OG10 varieties consists in  specializing the intermediate Jacobian fibration in the case where $W$ degenerates to the secant variety of the Veronese surface in $\bP^5$ (see Section \ref{subseclsv}).  There is however a serious obstruction to realize this program: starting with a well-understood or mild degeneration of cubic fourfolds $\mathcal W/\Delta$, the corresponding degeneration of the associated family
 of hyper-K\"ahler manifolds $\calX/\Delta$  can be  quite singular, and a priori hard to control. This is a common occurrence that can be already observed on the family
 of Fano varieties of lines of cubic fourfolds when the cubic acquires a node: mild degenerations of the cubic fourfold lead to families of associated hyper-K\"ahler manifolds $\calX/\Delta$ where  both the central fiber $X_0$ and the family $\calX$ are quite singular. So even if $X_0$ is  birational to a known hyper-K\"ahler manifold, due to the singularities, it is a priori  difficult to conclude that the general fiber $X_t$ is deformation
  equivalent to the given type. In \cite{lsv}, we avoided this issue following
   Beauville and Donagi \cite{bedo} by specializing to general Pfaffian cubics (for which
  our construction  has smooth specialization), while in \cite{debarrevoisin}, where another similar example was studied, an explicit resolution of the associated degeneration of hyper-K\"ahlers $\calX/\Delta$ was found. Theorem \ref{theo1} gives a uniform and simplified treatment of all these examples.

\begin{rema}{\rm  For the geometric applications we consider in this paper, checking finiteness of monodromy (see Theorem \ref{theo3}) is quite easy and
can be done directly, as we will explain case by case for completeness. This is due to the fact that we are considering (a family of) badly degenerating hyper-K\"ahler manifolds associated to (a family of) mildly degenerating Fano hypersurfaces, for which the finiteness of monodromy is clear.}
\end{rema}

In the final section, Section \ref{secdual}, we make some remarks on the degenerations of hyper-K\"ahler manifolds with infinite monodromy. Let's start by recalling the notion of {\it Type} for a degeneration.
\begin{definition}\label{defType}
Let $\calX^*/\Delta^*$ be a projective degeneration of hyper-K\"ahler manifolds (including the $K3$ case). Let $\nu\in\{1,2,3\}$ be the nilpotency index for the associated monodromy operator $N$ on $H^2(X_t)$ (i.e. $N=\log T_u$, where $T_u$  is the unipotent part of the monodromy $T=T_sT_u$). We say that the degeneration is of Type I, II, or III respectively if $\nu=1,2,3$ respectively.
\end{definition}

 For degenerations of $K3$ surfaces, a well known result (Theorem \ref{theokulikov2}) gives a precise classification of the central fiber of the degeneration depending on Type. Our results (esp. Theorem \ref{theo3}) give a strong generalization of the Type I case of this classification (i.e. finite monodromy implies the existence of smooth fillings). For the remaining Type II and III cases, we have weaker results, but which we believe to be of certain independent interest. Specifically, our focus is on the topology of the dual complex, a natural combinatorial gadget associated to semistable degenerations (or more generally {\it dlt degenerations}, by which we understand $(\calX,X_t)$ is dlt for every $t\in\Delta$, where dlt (divisorial log terminal) is as in \cite[Def. 2.7]{Kollar-Mori}; see also Appendix, esp. Def. \ref{defdlt}, for a brief review).
\begin{theo}\label{theodualcx}
Let $\calX/\Delta$ be a minimal dlt degeneration of $2n$-dimensional hyper-K\"ahler manifolds. Let $\Sigma$ denote the dual complex of the central fiber (and $|\Sigma|$ its topological realization). Then
\begin{itemize}
\item[(i)] $\dim |\Sigma|$ is $0$, $n$, or $2n$ iff the Type of the degeneration is I, II, or III respectively (i.e. $\dim |\Sigma|=(\nu-1)n$, where $\nu\in\{1,2,3\}$ is the nilpotency index of the log monodromy $N$).
\item[(ii)] If the degeneration is  of Type III, then $|\Sigma|$ is a simply connected closed pseudo-manifold, which is a rational homology $\bC\bP^n$.
\end{itemize}
\end{theo}

A few comments are in order here. First, this is clearly a (weak) generalization of Kulikov's theorem which states that for $K3$ surfaces,  $|\Sigma|$ is either a point, an interval, or $S^2$ depending on the Type of the degeneration.
Secondly, we note that under the assumption of minimal dlt degeneration, the dual complex is a well defined topological space (cf. \cite{dkx}, \cite{mn}, and \cite{nx}). There is a significant interest in the study of the dual complex $|\Sigma|$ in connection with the SYZ conjecture in mirror symmetry, especially in the context of the work of Kontsevich--Soibelman \cite{ks2,ks}.  In the strict Calabi-Yau case, it is expected that for maximal unipotent (MUM) degenerations $|\Sigma|$ is homeomorphic to the sphere $S^n$ (in any case, it is always a simply connected rational homology $S^n$).  The case $n=2$ follows from Kulikov's Theorem, and the cases $n=3$ and (assuming simple normal crossings) $n=4$ were confirmed recently by Koll\'ar--Xu \cite{kx}. Theorem \ref{theodualcx} follows by arguments similar to the Calabi-Yau case (esp. \cite{kx} and \cite{nx}) and a result of Verbitsky \cite{verbitskycoh}, which identifies the cohomology subalgebra generated by $H^2$ for a hyper-K\"ahler manifold. We also note that the occurrence of $\bC\bP^n$ in Theorem \ref{theodualcx} (see Theorem \ref{theostructureskeleton}
for a more general statement) is in line with the predictions of mirror symmetry. Namely, in the case of hyper-K\"ahler manifolds, the base of the Lagrangian fibration occurring in SYZ can be (conjecturally) identified with the base of an algebraic Lagrangian fibration, and thus expected to be $\bC\bP^n$ (see \S\ref{subssyz}--\ref{subshwang} for a discussion).

We close by noting that in passing (in our study of dual complexes for hyper-K\"ahler degenerations) we partially confirm a Conjecture of Nagai \cite{nagai} on the monodromy action on higher cohomology groups of hyper-K\"ahler manifolds (see Theorem \ref{theonagai} for a precise statement). Nagai has previously verified the conjecture for degenerations coming from Hilbert schemes of $K3$ surfaces or generalized Kummer varieties.

\subsection*{Acknowledgements} This project has started as a follow-up to \cite{lsv}. While preparing this manuscript, we have benefited from discussions with several people. We are particularly grateful to V. Alexeev, R. Friedman, K. Hulek, D. Huybrechts, J. Nicaise, and K. O'Grady for some specific comments related to parts of our paper.

The first author (JK) was partially supported by NSF grant DMS-1362960. He also acknowledges the support of Coll\`ege de France and FSMP that facilitated the meeting of the other co-authors and the work on this project. The second author (RL) was partially supported by NSF grants DMS-1254812 and DMS-1361143. He is also grateful to ENS and Olivier Debarre for hosting him during his sabbatical year, and to Simons Foundation and FSMP for fellowship support. The third author (GS) also would like to thank Coll\`ege de France for hospitality.

%%%%%%%%%%%%%%%%%%%%%%%%%%%%%%%%%%%%%%%
%%% S1 - Minimal Model
\section{Relative minimal models for degenerations of $K$-trivial varieties (aka Kulikov models)}\label{secmmp}

The Kulikov-Persson-Pinkham Theorem (\cite{kulikov}, \cite{pp}) states that any degeneration of $K3$ surfaces can be modified (after base change and birational transformations) to be semistable with trivial canonical bundle. In higher dimensions, the Minimal Model Program (MMP) guarantees for degenerations of $K$-trivial varieties the existence of a {\it minimal dlt model} $\calX/\Delta$ (i.e. $K_{\calX}\equiv 0$, and $(\calX,X_t)$ is dlt for any $t\in \Delta$; see Definition \ref{defdlt} for dlt).  The statement needed in this paper is due to Fujino \cite[Theorem I]{fujino-ss}. The following is a version of Fujino's theorem with a focus on the relationship between the central fiber of the original degeneration and the central fiber in the resulting minimal dlt model (in particular we note that any non-uniruled component will survive in the resulting minimal dlt model):
\begin{theo}\label{thmkulikovmodel} Let $f:X\to C$ be a projective morphism to a smooth, projective curve $C$. Assume that
\begin{enumerate}
\item[(i)] the generic  fiber $X_g$ is irreducible and birational to a $K$-trivial variety  with canonical singularities and
\item[(ii)] every fiber $X_c$ has at least one irreducible component $X_c^*$ that is not uniruled.
\end{enumerate}
Then there is a finite, possibly ramified, cover $\pi:B\to C$ and a projective morphism   $f':Y\to B$
with the following properties:
\begin{enumerate}\setcounter{enumi}{2}
\item[(0)] $Y$ is birational to $B\times_CX$ (and the birational map commutes with the projections to $B$),
\item[(1)]  the generic    fiber $Y_g$ is a $K$-trivial variety with terminal singularities,
\item[(2)] every fiber $Y_b$ is a $K$-trivial variety with canonical singularities, and
\item[(3)] if $X_c^*$ has multiplicity 1 in $X_c$ then $Y_b$ is birational to
 $X_c^*$ for $b\in \pi^{-1}(c)$.
\end{enumerate}
\end{theo}
\begin{rema}
As a simple consequence of this theorem, we see that there can be at most one non-uniruled component for the central fiber of a degeneration of $K$-trivial manifolds.
\end{rema}

\begin{proof}[Proof of Theorem   \ref{thmkulikovmodel}] By the semistable reduction theorem,  there is a finite ramified cover $\pi:B\to C$ such that $f_B:B\times_CX\to B$ is birational to a projective morphism
$q:Z\to B$ whose fibers are either smooth or reduced simple normal crossing divisors. Moreover, for $b\in \pi^{-1}(c)$ the fiber $Z_b$ has
at least one irreducible component $Z_b^*$ that admits a generically finite, dominant morphism  $\rho_b:Z_b^*\to X_c^*$. The degree of $\rho_b$ divides the   multiplicity of $X_c^*$ in $X_c$. Thus if the
 multiplicity is 1 then $\rho_b$ is birational.

By assumption the generic fiber is birational to a variety with
canonical singularities and semiample  canonical class.
(This is called a {\it good minimal model}; in our case some multiple of the  canonical class is actually trivial.)
Thus by  \cite[Thm.4.4]{lai-2009}
(see also \cite{fujino-ss})
the minimal model program  for $q:Z\to B$
terminates with a model  $f':Y\to B$ such that
 $Y$ has terminal singularities and
 $K_Y$ is $f'$-nef.

A general fiber of $f'$ has terminal singularities and
nef canonical class and it is also  birational to a K-trivial variety.
Thus general fibers of $f'$ are K-trivial varieties by \cite{kaw-flops} (see Cor. \ref{bir.to.CY.CY} for a precise statement).
 Therefore the canonical class $K_{Y/B}$ is
numerically equivalent to a linear combination of irreducible components of fibers. A linear combination of irreducible components of fibers is
nef iff it is numerically trivial  (hence a  linear combination of fibers).
Thus $K_{Y/B}$ is
numerically $f'$-trivial.

The  key point is to show that the fibers of $f':Y\to B$ are
irreducible  with canonical singularities. In order to do this, pick
$b\in B$. By assumption  $(Z, Z_b)$ is a  simple normal crossing (hence dlt) pair  and
$Z_b$ is numerically $q$-trivial. Thus every step of the
$K_Z$-minimal model program  for $q:Z\to B$ is also a step of the
$(K_Z+Z_b)$-minimal model program  for $q:Z\to B$. Thus
$(Y, Y_b)$ is dlt (cf.\  \cite[1.23]{kk-singbook}), in particular, every irreducible component of $Y_b$ is normal (cf.\ \cite[Sec.3.9]{fujino} or
\cite[4.16]{kk-singbook}).
 The exceptional divisors contracted by a
minimal model program  are uniruled by \cite[5-1-8]{KMM87}. Thus
$Z_b^*$ is not contracted and so it is birational to an
irreducible component $Y_b^*\subset Y_b$ which is therefore not uniruled.
Write $Y_b=Y_b^*+Y_b^{\circ}$
The adjunction formula  (cf.\ \cite[Sec.4.1]{kk-singbook})  now gives that
$$
K_{Y_b^*}\sim \bigl(K_{Y/B}+Y_b^*\bigr)|_{Y_b^*}\sim \bigl(K_{Y/B}-Y_b^{\circ}\bigr)|_{Y_b^*}\sim -Y_b^{\circ}|_{Y_b^*}.
$$
(Note that in general we could have an extra term coming from
singularities of $Y$ along a divisor of $Y_b^*$ but since  $(Y, Y_b)$ is dlt and $Y_b$ is Cartier,  this does not happen, cf.\ \cite[4.5.5]{kk-singbook}.)
If $Y_b^{\circ}\neq 0$ then  $-K_{Y_b^*}$ is effective and nonzero, hence
$Y_b^*$ is uniruled by \cite{mi-mo}; a contradiction.
Thus $Y_b=Y_b^*$ is irreducible. By the easy direction of the adjuction theorem
(cf.\ \cite[4.8]{kk-singbook}) it has only klt singularities and numerically trivial canonical class.

Let  $\tau_b:Y_b^c\to Y_b$ be the canonical modification  of $Y_b$
(cf.\ \cite[1.31]{kk-singbook}).
If $\tau_b$  contracts at least 1 divisor then
  $K_{Y_b^c}\sim \tau_b^*K_{Y_b}-E$ where $E$ is a positive linear combination of the $\tau_b$-exceptional divisors. As before, we get that
 $-K_{Y_b^c}$ is effective and nonzero, hence
$Y_b^c$ is uniruled by \cite{mi-mo}; a contradiction.

Thus $\tau_b$ is an isomorphism in codimension 1 and so
$K_{Y_b^c}\sim \tau_b^*K_{Y_b}$. Since $K_{Y_b^c}$ is $\tau_b$-ample,
this implies that $\tau_b$ is an isomorphism.  Hence
$Y_b$ has  canonical singularities, as claimed.
\end{proof}

\begin{rema} \label{general fiber unchanged} In general, the above construction gives a model $Y\to B$ whose general fibers are only birational to the corresponding fibers of  $X\to C$.
{\it We can modify the construction in order to leave the general fibers unchanged}. Assume first that  general fibers of $f$ are smooth K-trivial manifolds
over an open subset $C^0\subset C$. (This is the only case that we use in this note.)
We can then choose $B\times_CX\to Z$ to be an isomorphism
over $\pi^{-1}(C^0)$ and the minimal model program is then also an
isomorphism over   $\pi^{-1}(C^0)$. Thus we get
$f':Y\to B$ that is isomorphic to $f_B:B\times_CX\to B$
over $\pi^{-1}(C^0)$.

In general, assume that  the generic fiber of $f$ is a K-trivial variety with $\bQ$-factorial terminal singularities.
Let $C^0\subset C$ be an open subset
such that $f^{-1}(C^0)$ has  $\bQ$-factorial terminal singularities.
Let $D_P:=X\setminus f^{-1}(C^0)$,  with reduced structure.
First  construct a dlt modification
(cf.\ \cite[1.34]{kk-singbook}) of  $(X, D_P)$ to get
$(X', D'_P)\to C$ and then  pick any
$\pi: B\to C$ such that, for every $c\in P$, the multiplicities of all irreducible components of $X'_c$ divide the ramification index of
$\pi$ over $c$. After base-change and normalization we get a model
$q:Z\to B$ such that  $(Z, Z_b)$ is locally a quotient of a dlt pair  for every $b\in B$.
(See \cite[Sec.5]{dkx} for the precise  definition of such qdlt pairs  and their relevant properties.)
The rest of the proof now works as before to yield
$f':Y\to B$ that is isomorphic to $f_B:B\times_CX\to B$
over $\pi^{-1}(C^0)$.
\end{rema}

\begin{rema} \label{projective morphism rmk} In the above proof it is essential that $C$ be an algebraic curve. However, one can use \cite{knx} to extend the theorem to the cases when {\it $C$ is either a smooth Riemann surface or  a Noetherian, excellent, 1-dimensional, regular scheme  over a field of characteristic $0$}. However, even when $C$ is a smooth Riemann surface, we still  need to assume that $f:X\to C$ is at least locally projective, though this is unlikely to be necessary.
\end{rema}

The following results are contained in \cite{kaw-flops}, but not explicitly stated there.
For completeness, we state what has been used in the proof of the Theorem \ref{thmkulikovmodel} above.

\begin{lem}\label{lemflops} Let $X_i$ (for $i=1,2$) be projective varieties with canonical singularities and nef canonical classes. Let $p_i:Y\to X_i$ be birational morphisms. Then
$p_1^*K_{X_1}\simq p_2^*K_{X_2}$.
(That is, the birational map  $X_1\map X_2$ is crepant in the terminology of
 \cite[2.23]{kk-singbook}).
\end{lem}
\begin{proof}
We may assume that $Y$ is normal and projective.
Thus $K_Y\sim p_i^*K_{X_i}+E_i$ where $E_i$ is $p_i$-exceptional and effective since $X_i$ has canonical singularities. Thus
$E_1-E_2\simq  p_2^*K_{X_2}-p_1^*K_{X_1}$ is $p_1$-nef and
$-(p_1)_*(E_1-E_2)=(p_1)_*(E_2)$ is effective. Thus $-(E_1-E_2)$ is
effective by \cite[3.39]{km-book}.  Reversing the roles of $p_1, p_2$ gives that
$-(E_2-E_1)$ is
effective, hence $E_1=E_2$. \end{proof}

\begin{coro} \label{bir.to.CY.CY}
Let $X_i$ be birationally equivalent projective varieties with canonical singularities. Assume that $K_{X_1}\simq 0$ and $K_{X_2}$ is nef. Then
$K_{X_2}\simq 0$. \qed
\end{coro}

%%%%%%%%%%%%%%%%%%%%%%%%%%%%%%%%%%%%%%%%%%% S2 - Monodromy
\section{Cohomologically mild degenerations and Proof of Theorem \ref{theorr}}\label{secmono}
We first give an elementary proof of Theorem \ref{theorr}.
\begin{proof}[Proof of Theorem \ref{theorr}] Let $\mathcal{X}\rightarrow \Delta$ be as in
Theorem \ref{theorr}, with hyper-K\"ahler fibers of dimension $2n$. According to Theorem \ref{thmkulikovmodel}, completed by Remarks
\ref{general fiber unchanged} and \ref{projective morphism rmk}, we can find (after a finite  base change) a model  $\pi':\mathcal{X}'\rightarrow \Delta$  isomorphic
to $\mathcal{X}\rightarrow \Delta$ over $\Delta^*$ and such that the fiber
$X'_0$ has canonical singularities. The morphism $\pi'$ is projective and we can choose a
relative  embedding
$\mathcal{X'}\subset \Delta\times \mathbb{P}^N$. Let $H\subset \mathbb{P}^N$ be a general linear subspace of codimension $2n-2$. Then
$S_0:=H\cap X'_0$ is a surface with canonical singularities and for $t\not=0$,
$S_t:=H\cap X_t$ is smooth (after shrinking $\Delta$ if necessary).
The family
$\mathcal{S}/\Delta:=\mathcal{X}'\cap (\Delta\times H)$ is thus a family of surfaces with
smooth general fiber and central fiber with canonical singularities, hence the monodromy acting
on $H^2(S_t,\mathbb{Z}),\,t\not=0$ is finite. Indeed,  a family of surfaces over the disk  with central fiber having
canonical (or du Val)  singularies, can be simultaneously resolved after a finite base change.
On the other hand, for $t\not=0$, the restriction map
$$H^2(X_t,\mathbb{Z})\rightarrow H^2(S_t,\mathbb{Z})$$
is injective by hard Lefschetz, and commutes with the monodromy action. It follows that
the monodromy acting
on $H^2(X_t,\mathbb{Z})$, $t\not=0$, is also finite.
\end{proof}
\begin{rema} The same argument shows that in a projective degeneration
$\mathcal{X}\rightarrow \Delta$ with smooth general fiber and special fiber
$X_0$ satisfying ${\rm codim}\,( {\rm Sing}\,X_0)\geq k$, the monodromy acting on
$H^l(X_t,\mathbb{Z})$ is trivial for $l<k$. For $k=2$, we can also observe
that we only use the assumption that the central fiber is smooth in codimension one and du Val in codimension
$2$.
\end{rema}
We are now going to discuss the result above from the viewpoint of Hodge theory and
differential forms (similar arguments will be used again in Sections \ref{secsym} and \ref{secdual} below).
The standard tool for studying $1$-parameter degenerations $\calX/\Delta$ of K\"ahler manifolds is the Clemens--Schmid exact sequence (\cite{clemens}). Specifically, this establishes a tight connection between the mixed Hodge structure (MHS) of the central fiber and the limit mixed Hodge Structure (LMHS), which depends only on the smooth family (and not on the central fiber filling). As an application of this, under certain assumptions, one can determine the index of nilpotency for the monodromy $N=\log T$ for a degeneration purely in terms of the central fiber $X_0$ (e.g.  as an application of Kulikov-Persson-Pinkham Theorem and Clemens--Schmid exact sequence, one obtains the properness of the period map for $K3$ surfaces). The big disadvantage of the Clemens--Schmid sequence is that it assumes that $\calX/\Delta$ is a semistable family, which is difficult to achieve in practice. For surfaces, Mumford and Shah \cite{shahinsignificant} proved that one can allow $X_0$ to have mild singularities (called ``insignificant limit singularities'', which in modern terms is the same as Gorenstein semi-log-canonical (slc) singularities in dimension $2$) and still get a tight connection between the MHS on $X_0$ and the LMHS. Shah's method was based on constructing explicit semistable models for this type of singularities and reducing to Clemens--Schmid. Steenbrink noticed however that the true reason behind the close relationship between the MHS on the central fiber and the LMHS is the fact that Shah's insignificant singularities are du Bois. Specifically, we recall:
\begin{definition}[{Steenbrink \cite{steenbrink}}] \label{defist}
We say $X_0$ has {\it cohomologically insignificant singularities}, if for any $1$-parameter smoothing $\calX/\Delta$,  the natural specialization map
$$\spe_k:H^k(X_0)\to H^k_{\lim}$$
is an isomorphism on $I^{p,q}$-pieces (where $I^{p,q}$ denotes the Deligne's components of the MHS) with $p\cdot q=0$.
\end{definition}

\begin{theo}[{Steenbrink \cite{steenbrink}}]\label{theosteen}
If $X_0$ has du Bois singularities, then $X_0$ has cohomologically insignificant singularities.
\end{theo}

In other words, the original Shah \cite{shahinsignificant} theorem said that if  $X_0$ has insignificant limit singularities (equivalently Gorenstein slc in dimension 2) then $X_0$ has cohomologically insignificant singularities. While Steenbrink noticed that the correct chain of implications is actually:
 $${\rm insignificant \,\, limit \,\,singularities} \Rightarrow {\rm  du\,\, Bois\,\, singularities} \Rightarrow  {\rm cohomologically \,\,insignificant\,\, singularities}.$$
  Three decades later, coming from a different motivation, Koll\'ar--Kov\'acs \cite{kk} (building on previous work by Kov\'acs \cite{Kovacs} and others) have given a vast generalization of Shah's result:

\begin{theo}[{Koll\'ar--Kov\'acs \cite{kk}, \cite[6.32]{kk-singbook}}]\label{theokk}
Let $X_0$ be a variety with slc singularities. Then $X_0$ has du Bois singularities.
\end{theo}

We immediately get the following consequence which will be improved later on:
\begin{coro}\label{cororr}
Let $f:\calX \to \Delta$ be a projective morphism. Assume that the generic fiber $X_t$ is a smooth hyper-K\"ahler manifold and that the special fiber $X_0$ has canonical singularities and $H^{2,0}(X_0)\neq 0$. Then the monodromy acting on $H^2(X_t),\,t\not=0$, is finite.
\end{coro}
\begin{proof}
By  Theorem \ref{theokk}, the central fiber $X_0$, having canonical singularities (in fact, log canonical suffices), is du Bois. By Theorem  \ref{theosteen}, it follows that any degeneration is cohomologically insignificant, i.e. the natural specialization map $H^2(X_0)\to H^2_{\lim}$ is an isomorphism on the $I^{p,q}$ pieces with $p.q=0$. The assumption is that $I^{2,0}(H^2(X_0))\neq 0$. Since $\dim I^{2,0}_{\lim}+\dim I^{1,0}_{\lim}+\dim I^{0,0}_{\lim}=h^{2,0}(X_t)=1$, the only possibility is that $H^2(X_0)$ and $H^2_{\lim}$ are both pure and agree on the $(2,0)$ and $(0,2)$ parts. In other words,  $H^{2,0}_{\lim}$ and its complex conjugate are contained in the monodromy invariant
 part $H^2_{\textrm{inv}}$ of $H^{2}_{\lim}$. The Hodge structure on $H^2_{\textrm{inv}}$ is pure with
   $h^{2,0}=1$ and the restriction to  $H^2_{\textrm{inv}}$ of the  monodromy invariant pairing
 determined by the class $l$ of  a relatively ample line bundle is nondegenerate.
    By the Hodge index theorem,  $M_\mathbb{Z}=(H^2_{\textrm{inv}})^{\perp}\cap H^2_{\lim,\mathbb{Z}}\subset
    H^2_{\lim,\mathbb{Q}}$ is a negative definite lattice. In particular, $O(M_\mathbb{Z})$ is a finite group and, as  the monodromy action on $H^2_{\lim}$ factors up to a finite group through  $O(M_\mathbb{Z})$, it is finite.
\end{proof}

To strengthen the previous corollary, we consider the situation coming from the KSBA theory of compactifications of moduli. Namely, we are interested in degenerations (flat and proper) $\calX/\Delta$ which have the property that $K_{\calX}$ is $\bQ$-Gorenstein and the central fiber (is reduced and) has slc singularities. We call such  $\calX/\Delta$ a {\it KSBA degeneration}. If we assume additionally that $K_{\calX}$ is relatively nef, we call it a {\it minimal KSBA degeneration}. The total space of such a degeneration will have canonical singularities, and if needed, one can apply a terminalization, and obtain the so called {\it minimal dlt model}.

\begin{theo}\label{theomon1}
Let $\calX/\Delta$ be a projective degeneration of hyper-K\"ahler manifolds. Assume that $\calX/\Delta$ is a minimal KSBA degeneration (i.e. $K_{\calX}\equiv 0$ and $X_0$ is slc). Then the following are equivalent:
\begin{itemize}
\item[(1)] The monodromy action on $H^2(X_t)$ is finite.
\item[(2)] The special fiber $X_0$ has klt singularities (or equivalently, since Gorenstein degeneration, canonical singularities).
\item[(3)] The special fiber $X_0$ is irreducible and not uniruled (which in turn is equivalent to $X_0$ having a component that is not uniruled).
\end{itemize}
\end{theo}
\begin{rema}\label{remktirvmin}
The assumptions of $K$-triviality and minimality are clearly essential: a degeneration of curves to compact type has finite monodromy, but slc central fiber. Similarly, the blow-up of a family of elliptic curves gives a counterexample if we remove the minimality assumption.
\end{rema}

In the proof of Theorem \ref{theomon1}, as well as in Sections \ref{secsym} and \ref{secdual}, we will use the following result of Verbitsky.
\begin{theo}[{Verbitsky \cite[Thm 1.5]{verbitskycoh}, \cite{bogomolov}}]  \label{Verbitsky cohomology}
Let $X$ be a hyper--K\"ahler manifold of dimension $2n$. For  every $k=1, \dots, n$, the natural morphism
\[
\Sym^k H^2(X) \to H^{2k}(X)
\]
is injective.
\end{theo}

\begin{proof}[Proof of Theorem \ref{theomon1}]
The equivalence between (2) and (3) is due to Fujino \cite[Theorem II]{fujino-ss} (depending heavily on \cite{mi-mo}); see also the proof of Theorem \ref{thmkulikovmodel}.

The implication (2) $\Longrightarrow$ (1) is similar to Corollary \ref{cororr}. Namely, a variety $X_0$ with klt singularities has a pure Hodge structure on $H^1(X_0)$ and $H^2(X_0)$. This follows from the extension of holomorphic forms on such varieties (see \cite{gkkp}), and it is worked out in Schwald \cite{schwald}. Since, klt varieties are du Bois, it follows that the limit MHS in degree $2$ is pure. As before, this is equivalent to the monodromy being finite.

 To conclude the proof, it remains to see that if the monodromy is finite, the central fiber $X_0$ has to be klt. This follows from arguments given in Koll\'ar--Xu \cite[Sect. 4]{kx} and Halle--Nicaise \cite[\S3.3, esp. Thm. 3.3.3]{hn}. For completeness, we sketch the proof. Assume thus that $X_0$ is not klt.
After  possibly changing to a minimal dlt model, we see that $X_0$ has several irreducible components (since $X_0$ is not klt), and that each of the   components of $X_0$ are log Calabi-Yau varieties $(V,D)$ with $D\not \equiv0$. Since $D$ is an effective anti-canonical divisor, we conclude that $V$ is uniruled. On the other hand, by adjunction, we see that $D$ is a $K$-trivial variety.  We have two possibilities: either $D$ has canonical singularities, or $D$ is strictly log canonical. (For example, for a degenerations of $K3$ surfaces, the connected components of $D$ are either elliptic curves or cycles of rational curves. Furthermore, in the latter case, each irreducible component is rational with $2$ marked points, and thus log Calabi-Yau.)
Consider first the case that $D$ has canonical singularities.   For simplicity, we will further assume that $D$ is connected and thus irreducible. (In general, $D$ has at most two disconnected components \cite[\S32, \S16]{kx}, and the situation can be handled by similar arguments.) Under these assumptions,  there are two things to notice. First, $V$ has canonical singularities and it is uniruled,
and thus there is no top holomorphic form on it, giving $I^{2n,0}(H^{2n}(V))=0$. On the other hand, $D$ has canonical singularities, and it is $K$-trivial. Thus, $I^{2n-1,0}(H^{2n-1}(D))\neq 0$. Now the cohomology of $X_0$ is computed by a Mayer--Vietoris spectral sequence from the cohomology of the components $V$ and of the intersection strata $D$ (under our assumptions, there are only codimension $1$ strata).
In is immediate to see that $I^{2n,0} (H^{2n}(X_0))=0$ and $I^{2n-1,0} (H^{2n}(X_0))\neq 0$ (as in \ref{cororr}, it holds $\sum \dim I^{p,0} (H^{2n}(X_0))=1$). Using again slc $\Longrightarrow$ du Bois $\Longrightarrow$ cohomologically insignificant, we conclude $I^{2n,0}(H^{2n}_{\lim})=I^{2n,0}(H^{2n}(X_0))=0$. Using Verbitsky's Theorem \ref{Verbitsky cohomology}, it follows that the Hodge structure on $H^2_{\lim}$ is not pure (if $H^2_{\lim}$ were pure, then $I^{2,0} H^2_{\lim}\neq 0$, which in turn would give $I^{2n,0} H^{2n}_{\lim}\neq 0$, a contradiction). This means that the monodromy action on $H^2(X_t)$ is not finite, concluding the proof (under our assumptions on $D$).

The argument above carries through in the general case. The key point is that only the top dimensional components $V$ of (the natural stratification of) $X$ can contribute to $I^{2n,0}(H^{2n}(X_0))$. In fact, inductively, we can define $X_0^{[0]}$ to be the disjoint union of the components $V$, $X_0^{[1]}$ to be the disjoint union of the components of the double locus (or equivalently the log canonical center) $D$, and so on (cf. Appendix). Then, one can see that $I^{2n-k,0}(H^{2n}(X_0))\neq 0$ precisely for the deepest codimension $k$ stratum $X_0^{[k]}\neq \emptyset$. Namely, as before, $X_0^{[k]}$ is $K$-trivial with canonical singularities, while the higher dimensional strata $X_0^{[l]}$ ($l<k$) are uniruled. Then, the claim follows via  a spectral sequence analysis as is \cite[Claim 32.1]{kx}  and  \cite[(32.2)]{kx}.  (Since we are interested only in the holomorphic part of the cohomology,  one can work as if $X_0$ is simple normal crossings. The precise statements in  the dlt situation are discussed in Appendix \ref{dltpairs}, see esp. Cor. \ref{spectral sequence}.)
\end{proof}

 We conclude this section with another proof of Theorem \ref{theorr}.

\begin{proof}[Second proof of Theorem \ref{theorr}]
By Theorem \ref{thmkulikovmodel}, after finite base change and birational transformations, we arrive at a minimal dlt model such that the central fiber consists of a unique non-uniruled component with canonical singularities. By  Theorem \ref{theomon1}, the monodromy acting on $H^2(X_t)$ is finite.
\end{proof}

%%%%%%%%%%%%%%%%%%%%%%%%%%%%%%%%%%%%%%%%
%%% S3 - Proof1
\section{Proof of Theorems \ref{theo1},  \ref{theo2} and \ref{theo3}}\label{secmain}
\begin{proof}[Proof of Theorems \ref{theo2} and  \ref{theo3}]  Let $\mathcal{X}\rightarrow \Delta$ be a projective morphism with smooth hyper-K\"ahler fibers over $\Delta^*$, satisfying the hypothesis
of Theorem \ref{theo2}. By assumption, one component of the central fiber  is not uniruled.  By Theorem  \ref{theorr}, after performing a  finite base change,  we can assume
  that the monodromy acting on $H^2(X_t),\,t\not=0$ is trivial.
We are now reduced to the situation of Theorem \ref{theo3}. Using a relatively ample line bundle on
$\mathcal{X}\rightarrow \Delta$, the fibers $X_t$  are projective with a given polarization
$l:=c_1(\mathcal{L}_{\mid X_t})$.
Let $q$ be the Beauville-Bogomolov form on $H^2(X_{t_0},\mathbb{Q})$ for some given
$t_0\in \Delta^*$ and
let $$\mathcal{D}_l=\{\eta\in \mathbb{P}(H^2(X_{t_0},\mathbb{C})),\,q(\eta)=0,\,q(\eta,\overline{\eta})>0,\,q(\eta,l)=0\}$$
be the polarized period domain for deformations of $(X_{t_0},l)$.
The monodromy being trivial, the period map $\mathcal{P}^*:\Delta^*\rightarrow \mathcal{D}_l$
is  well defined and by \cite{Gri}, it extends to a holomorphic map
$\mathcal{P}:\Delta\rightarrow \mathcal{D}_l$. (Note that  this is one  place where we seriously use the projectivity assumption. Griffiths' extension theorem only holds for polarized period maps.)

By \cite{huybrechts}, the unpolarized  period map is surjective from any
connected component of the marked deformation space of $X_t$ to $\mathcal{D}_l$.
Thus there is a hyper-K\"ahler manifold $X'_0$ which is deformation equivalent
to $X_t$,  with period point
$\mathcal{P}(0)\in\mathcal{D}_l$. Finally, as $q(l)>0$, $X'_0$ is projective by \cite{huybrechts}.
The local period map $B_l\rightarrow \mathcal{D}_l$ is a local holomorphic diffeomorphism,
where $B_l$ is a ball in the universal deformation space
 of the pair $(X'_0,l)$ consisting
of a hyper-K\"ahler manifold and a degree $2$ Hodge class on it, and thus
the holomorphic disk $\mathcal{P}:\Delta\rightarrow\mathcal{D}_l$ can be seen (after shrinking) as a holomorphic disk
$\Delta\rightarrow B_l$. Shrinking $B_l$ and $\Delta$ if necessary and restricting to $\Delta$ the universal family $\mathcal{X}_{\textrm{univ}}\rightarrow B_l$
(which exists over $B_l$), there is a family
$\mathcal{X}'\rightarrow \Delta$ of marked hyper-K\"ahler
manifolds in the same deformation class as $X_t$. Furthermore, by construction, the associated period map
$\mathcal{P}'$ can be identified with $\mathcal{P}$.

We now apply Verbitsky's Torelli theorem \cite{verbi}, which allows us to conclude that
for any $t\in \Delta^*$, $X'_t$ and $X_t$ are birational.
Furthermore, as the family $\pi':\mathcal{X}'\rightarrow \Delta$ is smooth proper
with K\"ahler fibers, there exists a $\mathcal{C}^\infty$ $(1,1)$-form on $\mathcal{X}$ whose restriction to the fibers $X_t$ are K\"ahler, hence provides a $\mathcal{C}^\infty$
family
$(\omega_t)_{t\in \Delta_t}$ of K\"ahler classes in the fibers of
$\pi'$. Moreover, we also know that  the morphism
$\pi:\mathcal{X}\rightarrow \Delta$ is projective. It follows (see
\cite{bishop}) that
the  relative Douady space over $\Delta$ (analytic version of the relative Hilbert scheme)
of subschemes in fibers of $X'_t\times X_t$
is a countable union of analytic varieties which are proper over $\Delta$.
Furthermore, note that for each component
$S$ of this relative Douady space with corresponding family $\Gamma_S\rightarrow S\rightarrow \Delta,\,\Gamma_S\subset \mathcal{X}_S\times_S\mathcal{X}'_S$, with
$\mathcal{X}_S=\mathcal{X}\times_\Delta S,\,\mathcal{X}'_S=\mathcal{X}'\times_\Delta S$, the property
that $\Gamma_{S,t}$ is the graph of a birational map between $X_s$ and $X'_s$
is Zariski open  in $S$.  Finally, graphs of birational maps between two hyper-K\"ahler manifolds are rigid, so each such component $S$ containing at least one graph
 of a birational map between smooth fibers $X_s$ and $X'_s$  has dimension either $0$ or  $1$. Of course, the union of components of dimension $0$ provide only countably many points in
 $\Delta^*$. Thus,
we conclude that $\Delta^*$ minus countably many points is the union of
the images of the maps $S^0\rightarrow \Delta^*$, over the countably
many  $1$-dimensional components
$S$  admitting a dense Zariski open subset
$S^0$ over $\Delta^*$ such that the cycles $\Gamma_{S,s}$ parameterized by $s\in S^0$ are
graphs of birational maps between the fibers of both families.  Hence there exists such an $S$ which
dominates $\Delta$. We may assume that $S$ is smooth and,    by properness,  that the map
$S\rightarrow \Delta$ is finite and surjective.
The universal subvariety $\Gamma_S\subset \mathcal{X}_S\times_S\mathcal{X}'_S$ provides the desired
fibered birational isomorphism.
\end{proof}
\begin{rema}{\rm The arguments given here are very similar to those used
in \cite{burnsrapoport} and even  simpler since we
have Verbitsky's theorem, while Burns and Rapoport use them to prove Torelli's theorem
for $K3$ surfaces.
}
\end{rema}
Let us note the following consequence of Theorem \ref{theo2}.
\begin{coro}\label{corofinitemonoevrywhere}  Assumptions as in Theorem
\ref{theo2}. The monodromy action on $H^k(X_t)$ is finite for any $k$.
\end{coro}
\begin{rema} Note that this corollary  is not a trivial consequence of Theorem \ref{theo2} because Theorem \ref{theo2} does not say that
the original family, after pullback, can be filled-in with a smooth central fiber. It just says that this can be done after pullback and replacing the family by another one with bimeromorphic and hyper-K\"ahler fibers.
\end{rema}
\begin{proof}[Proof of Corollary \ref{corofinitemonoevrywhere}]  Huybrechts' theorem \ref{theohuy} tells us that $X_t$ and $X'_t$ are deformation equivalent. It also says a little more: for any $t\in \Delta^*$, there exists
a cycle $\Gamma_t$ in $X_t\times X'_t$ which is a limit of graphs
of isomorphisms between deformations of $X_t$ and $X'_t$
and thus induces an isomorphism of cohomology rings
\begin{eqnarray}\label{eqisoavion} H^*(X_t,\mathbb{Z})\cong H^*(X'_t,\mathbb{Z}).
\end{eqnarray}
As the two families are K\"ahler over $\Delta$, we can use properness of the relative Douady spaces to conclude that possibly after base change, there exists a cycle
$\Gamma\in \mathcal{X}\times_\Delta\mathcal{X}'$ whose restriction $\Gamma_t$ induces
the isomorphism (\ref{eqisoavion}). The monodromy action on $H^k(X_t)$ thus becomes
 trivial  after base change for any $k$.
\end{proof}
\begin{proof}[Proof of Theorem \ref{theo1}] The proof follows closely
the proof of Huybrechts \cite[Theorem 4.6]{huybrechts}. Under the assumptions of Theorem
\ref{theo1}, Theorem \ref{theo2} gives us
a birational map $\phi:\mathcal{X}'_S\dashrightarrow \mathcal{X}_S$ over a finite cover
$S$ of $\Delta$, where $\mathcal{X}'_S$ is smooth over $S$ with hyper-K\"ahler fibers.
Let us blow-up $\mathcal{X}_S$ until it becomes smooth, say
$\widetilde{\mathcal{X}}_S$, and then let us blow-up $\mathcal{X}'_S$ successively along smooth centers until
the rational map $\phi$ induces a morphism
$\tilde{\phi}:\widetilde{\mathcal{X}}'_S\rightarrow \widetilde{\mathcal{X}}_S$ over $S$. By assumption, the central fiber
$X_0$ of our original family has a multiplicity $1$ component $V$ which is birational to the smooth hyper-K\"ahler manifold
$Z'_0$ (which is projective, as it is Moishezon and K\"ahler). The proper transform $\widetilde{V}$ of $V$ is thus birational to $V$ and also appears
as a multiplicity $1$ component of the central fiber of $\widetilde{\mathcal{X}}_S\rightarrow S$.
(It is at this point that we use the fact that $V$ is a multiplicity $1$ component of
$X_0$; otherwise the desingularization process needed to produce $\mathcal{X}'_S$
can involve a normalization which replaces $V$ by a generically finite cover of it.)
As $\tilde{\phi}$ is proper and birational, exactly one component $V'$ of the central
fiber of $\widetilde{\mathcal{X}}'_S\rightarrow S$ maps onto $\widetilde{V}$
 and the morphism $V'\rightarrow \widetilde{V}$ is birational.
 Hence $V'$ is birational to $Z'_0$. On the other hand, as ${\mathcal{X}}'_S$ is smooth,
 all the exceptional divisors of $\widetilde{\mathcal{X}}'_S\rightarrow \mathcal{X}'_S$
 are uniruled, and thus the only component  of the central fiber which can be birational to a hyper-K\"ahler
 manifold is the proper transform of $X'_0$ (via $\widetilde{\mathcal{X}}'_S\rightarrow \mathcal{X}'_S$). It follows that $V'$ is birational to $X'_0$. Thus we proved that
 $X'_0$ and $Z'_0$ are birational.
By Huybrechts' theorem \cite{huybrechts} (Theorem \ref{theohuy}), it follows that
the two hyper-K\"ahler manifolds $X'_0$ and $Z'_0$  are deformation equivalent.
On the other hand, $X'_0$ is by definition deformation equivalent
to $X'_t$ which is birational to $X_t$ for $t\not=0$, hence is deformation equivalent
to $X_t$ by  Huybrechts' theorem again, since $X_t$ and $X'_t$ are smooth. We conclude that $X_t$ is deformation equivalent to $Z_0'$ as claimed.
\end{proof}

%%%%%%%%%%%%%%%%%%%%%%%%%%%%%%%%%%%%%%%%
%%% S4 - Proof2
\section{Symplectic singularities; Alternative proof to Theorems \ref{theo2} and \ref{theo3}}\label{secsym}
As previously discussed, the filling Theorems \ref{theo2} and \ref{theo3} are results specific to hyper-K\"ahler manifolds (see also Remark \ref{remfinitemoncy}). In the previous section, we proved these theorems by using the deep results (Torelli and surjectivity) for the period map for hyper-K\"ahler manifolds due to Verbitsky \cite{verbi} and Huybrechts \cite{huybrechts}. The MMP results are only tangentially used. In this section, we give an alternative proof to Theorems \ref{theo2} and \ref{theo3} relying  on the  MMP results  of Section \ref{secmmp} as well as on results already used in Section \ref{secmono}. The key point of this alternative proof is to notice that in the case of a minimal dlt degeneration $\calX/\Delta$ of hyper-K\"ahler manifolds with finite monodromy, the central fiber $X_0$ has {\it symplectic singularities} in the sense of Beauville \cite{Beauville}. The filling theorems now follow from  the results of Namikawa \cite{namikawa1, namikawa2}, which roughly say that the symplectic singularities are rigid (and thus, if $X_0$ is not smooth, there is no smoothing).

\begin{definition} \label{symplecticres}
A a variety $Y$ with canonical singularities is called a \emph{symplectic variety} in the sense of Beauville \cite{Beauville} if the smooth locus of $Y$ carries a holomorphic symplectic form with the property that it extends to a holomorphic form on any resolution of $Y$. A resolution $f: \wt Y \to Y$ is called \emph{symplectic} if the symplectic form on the smooth locus of $Y$ extends to a global holomorphic symplectic form on $\wt Y$. (Note that a symplectic resolution is, in particular, crepant.)
\end{definition}

\begin{proof}[Another proof of Theorems \ref{theo2} and \ref{theo3}] We start as in the proof of Theorem \ref{theorr}. By Remark \ref{projective morphism rmk} we can apply Theorem \ref{thmkulikovmodel} to the projective morphism $f: \mc X  \to \Delta$.
This gives,
possibly after a base change $\pi: \Delta \to \Delta$, a projective morphism $\mc Y \to \Delta$ and a birational map $h: \mc Y \dashrightarrow \mc X \times_\Delta \Delta$ which induces a birational map from the central fiber $Y_0$, which is a K-trivial variety with canonical singularities, to  $X_0^*$. By Remark \ref{general fiber unchanged}, we can also ensure that $h$ induces an isomorphism between the fibers $Y_t $ and $X_{\pi(t)}$, for $t \neq 0$.

We claim that $Y_0$ is {\it a symplectic variety} in the sense of Beauville. As already noted $Y_0$ has canonical singularities. To check that the smooth locus of $Y_0$ carries a holomorphic $2$-form that is symplectic and that it extends to resolutions, we use  arguments similar to those of Section \ref{secmono}, but some extra care is needed to be able to interpret $H^{2,0}(Y_0)$ as holomorphic forms (N.B. $Y_0$ is singular). To start, since $Y_0$ has canonical singularities then $Y_0$ has rational singularities (\cite[2.77 and 2.88]{kk-singbook}) and then du Bois singularities (\cite{Kovacs}; this also follows from \cite{kk}, see Theorem \ref{theokk}).
By \cite[Thms 1 and 2]{steenbrink} this implies that $R^{i}f_* \mc O_{\mc Y'}$ is locally free (of rank one if $i$ is even, zero otherwise) and satisfies base change, that the Hodge filtration on $H^{i}(Y_0)$ satisfies $\Gr^0_FH^{i}(Y_0)=H^{i}( Y_0, \mc O_{Y_0})$, and that the degeneration $\mc Y \to \Delta$ is cohomologically insignificant (cf. Section \ref{secmono}), i.e. that the specialization map $H^{i}(Y_0) \stackrel{\spe}{\longrightarrow} H^{i}_{\lim}$
induces an isomorphism on the $(p,q)$--pieces with $p\cdot q=0$. Since $Y_0$ has rational singularities, $\pi^*: H^{2}(Y_0) \to H^{2}(\wt{Y}_0)$ is injective (\cite[(12.1.3.2)]{Kollar-Mori}) and hence by \cite[Cor. 8.2.5]{DeligneIII} the MHS on $H^2(Y_0)$ is pure of weight two. In particular, $\Gr^0_FH^{2}(Y_0)=H^{0,2}(Y_0)$ and hence $h^{2,0}(Y_0)=h^{0,2}(Y_0)=1$. Let $\sigma_0 \in H^2(Y_0)$ be a generator of $H^{2,0}(Y_0)=F^2 H^2(Y_0)$. We need to show that $\sigma_0$ defines a holomorphic symplectic form on the smooth locus of $Y_0$, which extends to a holomorphic $2$--form on any resolution $ \pi: \wt {Y}_0 \to  Y_0$. We remarked that the pullback is injective on degree two cohomology, so $\pi^*(\sigma_0)$ defines a non--zero holomorphic $2$--form $\wt \sigma_0$ on $\wt {Y}_0$. To show that it is generically symplectic, it is sufficient to show that $\wt \sigma_0^n \neq 0$. The cup--product is compatible with the specialization map and also with Deligne's MHS (\cite[Cor. 8.2.11]{DeligneIII}), so $\spe( \sigma_0^n )=\spe(\sigma_0)^n $ lies in $F^{2n} H^{2n}_{\lim} \cap W_{2n}=H^{2n,0}_{\lim}$. Since $H^i_{\lim}$ is the $i$--th cohomology of a smooth hyper--K\"ahler manifold, by the result of Verbitsky on the cohomology of hyper-K\"ahler manifolds (Theorem \ref{Verbitsky cohomology})
we know that $\Sym^i H^2_{\lim} \rightarrow H^{2i}_{\lim}$ is injective for $i \le n$. Hence $\spe( \sigma_0^n)$ generates $H^{2n,0}_{\lim}$ and, in particular, $\sigma_0^n$ is non--zero. We are left with showing that the pullback $\wt \sigma_0^n=\pi^*(\sigma_0^n) \neq 0$.
But this follows from the fact that the pullback morphism $\pi^*: H^{i}(Y_0) \to H^{i}(\widetilde Y_0)$ is injective on the weight $i$ part of the MHS (\cite[Cor. 8.2.5]{DeligneIII}). Since $Y_0$ is K-trivial with canonical singularities, the vanishing locus of the holomorphic section $\wt \sigma_0^n$ of the canonical bundle of $\wt Y_0$ is an effective divisor, supported on the exceptional locus of $ \pi: \wt {Y}_0 \to  Y_0$. Hence,  $\wt \sigma_0$ is a holomorphic two form on $\wt Y_0$, which is nondegenerate (i.e. symplectic) at least on an open set containing the locus where $\pi$ is an isomorphism.

Let $\pi: M \to Y_0$ be a $\mathbb Q$--factorial terminalization, i.e., $M$ is $\mathbb Q$--factorial and terminal, and  $\pi$ is a crepant morphism. This always exists by \cite[Cor. 1.4.3]{BCHM}. We use Namikawa's result (\cite[Cor 2]{namikawa2}) to show that $M$ is in fact smooth.
For the readers sake, we recall Namikawa's argument: The Main Theorem in \cite{namikawa2} shows that $\mathbb Q$--factorial symplectic varieties with terminal singularities are locally rigid. Hence, to prove that $M$ is smooth, it is enough to show that a smoothing of $Y_0$ determines a smoothing of $M$. Indeed, since $R^1 \pi_* \mc O_{M}=0$,   any deformation of $M$  induces a deformation of $Y_0$ (\cite[Thm 1.4]{Wahl}, \cite[11.4]{Kollar-Mori}). More specifically, let $\mc M_{\Def} \to \Def(M)$ and $\mc{ Y}_{\Def} \to \Def(Y_0)$ be versal deformation spaces for $M$ and $Y_0$, respectively. By \cite[Thm 1]{namikawa2} there is a finite surjective morphism $\pi_*:\Def(M) \to \Def(Y_0) $ which lifts to a morphism $\Pi_*:  \mc M_{\Def} \to \mc{ Y}_{\Def}$ inducing an isomorphism between the general deformation of $M$ and the general deformation of $Y_0$. Hence $M$ is smooth and moreover, by \cite[Thm 2.2]{namikawa1}, any smoothing of $Y_0$ is obtained as a flat deformation of $M$.
Up to a base change, we can thus lift the morphism $\Delta \to \Def(Y_0)$ associated to the family $\mc Y\to \Delta$ and get a morphism $\Delta \to  \Def(M) $, which we use to pull back the universal family. We thus have two one-parameter deformations $\mc M \to \Delta$ and $\mc Y \to \Delta$, of $M$ and of $Y_0$, respectively, together with a morphism $\mc M \to \mc Y$ over $\Delta$, which induces an isomorphism away from the central fiber and the symplectic resolution $M \to Y_0$ over the origin.
\end{proof}

\begin{coro}\label{coromono2}
Let $\mathcal{X} \to \Delta$ be a projective degeneration with general fiber $X_t$ a smooth hyper-K\"ahler manifold. Assume that one irreducible component $V$ of the central fiber $X_0$ is not uniruled and appears with multiplicity one. Then any minimal model of $V$ has a symplectic resolution (which is a smooth hyper-K\"ahler deformation equivalent to $X_t$) and the monodromy action on the cohomology of a smooth fiber of $f$ is finite. Conversely, if the monodromy of $\mathcal{X} \to \Delta$ is finite, then there exists a smooth family $\mathcal{Y}\to \Delta$ of hyper-K\"ahler manifolds that is isomorphic over $\Delta^*$ to (a finite base change) of $\mathcal{X}^*\to \Delta^*$.
\end{coro}
\begin{proof}
The first part of the statement follows directly from the above proof. The second statement follows by using Remark \ref{general fiber unchanged}, the equivalence  $(1) \iff (3)$ of Theorem \ref{theomon1}, and then again the arguments of this section.
\end{proof}

\begin{rema}
Notice that the second statement gives a slightly stronger version of Theorem \ref{theo3} (in that it leaves the general fibers unchanged).
\end{rema}

Notice that in the course of the proof above we have shown a special case of the following observation of  Greb--Lehn--Rollenske \cite[Prop. 6.4]{glr} (whose proof also relies on \cite{namikawa2}).

\begin{rema}
If $X_0$  is symplectic variety which is birational to a smooth hyper-K\"ahler manifold, then $X_0$ admits a symplectic resolution.
\end{rema}

%%%%%%%%%%%%%%%%%%%%%%%%%%%%%%%%%%%%%%%%
%%% S5 - Application
\section{Application: Deformation type of hyper-K\"ahler manifolds via degeneration methods}\label{secexamples}
The main tool available for constructing hyper-K\"ahler manifolds is Mukai's method, namely starting with a K3 or  an abelian surface and considering moduli spaces of sheaves on them. This leads to $K3^{[n]}$ type and also after a delicate desingularization process, to the exceptional  OG10 examples (and, similarly, to the generalized Kummer varieties and the exceptional OG6 manifolds
 when starting from an abelian surface). It turns out that there are other geometric constructions leading to hyper-K\"ahler manifolds, most notably starting with a cubic fourfold (\cite{bedo}, \cite{lsv}, \cite{llsvs}). In all these cases, a series of ad hoc geometric arguments were used to establish the deformation equivalence of these new constructions to the Beauville--Mukai examples.
As an application of our results on degenerations of hyper-K\"ahler manifolds, we give in this section a somewhat unified and simplified method to obtain their belonging to a given deformation class. Namely, as investigated by Hassett \cite{hassett}, various codimension $1$ loci (denoted $\calC_d$) in the moduli of cubic fourfolds are Hodge theoretically (and sometimes geometrically) related to $K3$ surfaces. Specializing to these loci often gives a clear link between the hyper-K\"ahler manifolds constructed from cubics and the ones constructed from $K3$ surfaces by Beauville--Mukai or O'Grady constructions.  In fact the easiest specializations of a cubic fourfold  linking cubic fourfolds  to   $K3$ surfaces  are specializations to nodal cubic fourfolds (the divisor $\calC_6$ in Hassett's notation) or  degenerations to the cubic secant to the Veronese surface in $\bP^5$ (which give the divisor $\calC_2$; see  \cite{laza}). In these cases, the associated $K3$ surface is obvious, and after specialization, a birational model of  the associated hyper-K\"ahler manifold is easy to understand. The problem is that even when the  degeneration of the cubic is as mild as possible, the associated  hyper-K\"ahler manifolds (e.g. the Fano variety of lines) specialize to quite singular objects. Our main result Theorem \ref{theo1} tells us that as long as the holomorphic $2$-form survives in the degeneration, we can ignore the singularities of the central fiber in order to compute the deformation type.
In this section, we are thus going to revisit \cite{bedo}, \cite{debarrevoisin}, \cite{lsv}, and \cite{addingtonlehn} in the light of Theorem \ref{theo1}.
\subsection{Fano variety of lines of a cubic fourfold} \label{subsecbedo}
Let $X$ be a smooth cubic fourfold. The variety of lines $F(X)$ is a smooth projective hyper-K\"ahler fourfold by \cite{bedo}. It is deformation equivalent to $S^{[2]}$, where $S$ is
$K3$ surface. More precisely, Beauville and Donagi prove the following:
\begin{theo}\label{theobedo} Let $X$ be a smooth  Pfaffian cubic fourfold and $S$ be the associated $K3$ surface. Then $F(X)$ is isomorphic to $S^{[2]}$.
\end{theo}
Here a Pfaffian cubic fourfold is defined as the intersection of the Pfaffian
cubic in $\mathbb{P}^{14}=\mathbb{P}(\bigwedge^2V_6)$ with a $\mathbb{P}^5=\mathbb{P}(W_6)\subset
\mathbb{P}^{14}$. The associated $K3$ surface
$S$ is defined in the Grassmannian
$G(2,V_6^*)$ by the space $W_6$, seen as a set of Pl\"ucker linear forms on
$G(2,V_6^*)$.

Note that Theorem \ref{theobedo} is used in \cite{bedo} in order to prove that $F(X)$
is a smooth projective hyper-K\"ahler
fourfold for general $X$. However, this last fact can be seen directly by saying that (1) $F(X)$ is smooth as all varieties of lines of smooth cubics are; (2)  $F(X)$ has trivial canonical bundle
as it is the zero set of a transverse section of $S^3\mathcal{E}$ on $G(2,6)$, where
$\mathcal{E}$ is the dual of  tautological rank $2$  vector subbundle on $G(2,6)$, and (3)
$F(X)$ has a holomorphic $2$-form defined as
$I^*\alpha_X$, where $I\subset F(X)\times X$ is the incidence correspondence,
and it can easily be proved to be generically nondegenerate.

Instead of considering the specialization to the Pfaffian case, let us consider the specialization
to the nodal case, where $X$ specializes to $X_0$ with one ordinary
double point at $0\in X_0$.
Let $\pi^X:\mathcal{X}\rightarrow \Delta$ be such a Lefschetz degeneration,
and let $\pi^F:\mathcal{F}\rightarrow \Delta$ be the associated family
of Fano varieties of lines.
 It is well-known
 (see \cite{CG}) that $F(X_0)$ is birational to
 $\Sigma^{[2]}$, where
 $\Sigma $ is the surface of lines in $X_0$ passing through
 $0$. $\Sigma$ is the smooth intersection of a quadric and a cubic in
 $\mathbb{P}^4$, hence a $K3$ surface, and the birational map
 $\Sigma^{[2]}\dashrightarrow F(X_0)$ associates to a pair of lines
 $l,\,l'$ through
 $0$ the residual line of the intersection $P_{l,l'}\cap X_0$ where
 $P_{l,l'}$ is the plane generated by $l$ and $l'$.
Note also that the variety of lines of $X_0$ is  smooth away from the surface $\Sigma$,
hence $F(X_0)$ is a multiplicity $1$ component of the central fiber of the family $\mathcal{F}\rightarrow \Delta$.
 Theorem \ref{theo1} thus applies showing that $F(X_t)$ is deformation equivalent
 to $\Sigma^{[2]}$.
 \begin{rema}{Note that in this example, we can check directly that
 the monodromy acting on $H^2(F(X_t))$ is finite (thus avoiding the use of Theorems
  \ref{thmkulikovmodel} and  \ref{theorr}). Indeed, the monodromy action on $H^4(X_t)$ is finite, being given by a Picard-Lefschetz reflection,
and as the relative incidence correspondence
$\mathcal{I}\subset \mathcal{F}\times_\Delta \mathcal{X}$ induces an isomorphism
$\mathcal{I}^*:R^4\pi^X_*\mathbb{Q}\rightarrow  R^2\pi^F_*\mathbb{Q}$
of local systems over $\Delta^*$,
the monodromy acting on
 $H^2(F(X_t))$ is also finite. The same remark applies in fact to all the cases described in Section
 \ref{secexamples}.
 }
 \end{rema}
\subsection{Debarre-Voisin hyper-K\"ahler fourfolds}
The hyper-K\"ahler fourfolds constructed in \cite{debarrevoisin} are defined as
zero-sets $Y_\sigma$ of general
sections $\sigma\in
\bigwedge^3V_{10}^*$ of the rank $20$ vector bundle $\bigwedge^3\mathcal{E}$ on the Grassmannian $G(6,V_{10})$, where
$\mathcal{E}$ is the dual of the rank $6$ tautological  vector subbundle on $G(6,10)$.
It is proved in \cite{debarrevoisin} that these varieties are deformation equivalent
to $K3^{[2]}$.
Let us now explain how the use of Theorem \ref{theo1} greatly simplifies the proof of this statement.
The choice of $\sigma\in \bigwedge^3V_{10}^*$ also determines a hypersurface (a Pl\"ucker
hyperplane section)
$X_\sigma\subset G(3,V_{10})$. For general $\sigma$, $X_\sigma$ is smooth of dimension
$20$ and
there is an isomorphism
\begin{eqnarray}\label{eqingcoh}
G_\sigma^*: H^{20}(X_\sigma,\mathbb{Q})_{prim}\rightarrow H^2(Y_\sigma,\mathbb{Q})_{prim}
\end{eqnarray}
induced by the incidence correspondence $G_\sigma\subset Y_\sigma\times X_\sigma$,
where the fiber of $G_\sigma$ over a point $[W_6]\in Y_\sigma$ is
the Grassmannian $G(3,W_6)$ which is by definition contained in
$X_\sigma$ (see \cite{debarrevoisin}).
In the paper \cite{debarrevoisin}, the generic nodal  degeneration
$\pi^X:\mathcal{X}\rightarrow \Delta$ of $X_\sigma$ is considered, with the associated
family $\pi^Y:\mathcal{Y}\rightarrow \Delta$ and relative incidence correspondence
$\mathcal{G}\subset \mathcal{Y}\times_\Delta\mathcal{X}$.
We have the following result (see \cite[Theorem 3.3]{debarrevoisin}):
\begin{theo} \label{theodev}The variety $Y_{\sigma_0}$ is reduced and birationally equivalent to
$S^{[2]}$, where $S$ is a $K3$ surface.
\end{theo}
We are thus in position to apply Theorem \ref{theo1} and this shows that
the smooth fibers  $\mathcal{Y}_{\sigma_t}$ are deformation equivalent
to $K3^{[2]}$.
In the paper \cite{debarrevoisin}, the proof of this fact used a delicate analysis of
the pull-back to $S^{[2]}$ of the Pl\"ucker line bundle, so as to apply a
Proj argument  in the spirit of Huybrechts.
For the sake of completeness, let us recall how  the $K3$ surface
$S$ is constructed
in this case.
Let $X_\sigma$ be singular at $[W]\in G(3,V_{10})$. Then $\sigma_{\mid W}=0$ in $\bigwedge^3W^*$ and furthermore
$\sigma$ vanishes in $\bigwedge^2W^*\otimes (V_{10}/W)^*$. Thus $\sigma$ defines
an element $\sigma_2$ of $W^*\otimes \bigwedge^2(V_{10}/W)^*$. Let $V_7:=V_{10}/W$.
The  surface $S$ is defined as the set of $3$-dimensional subspaces
of $V_7$ whose inverse image in $V_{10}$ belongs to $Y_\sigma$. This is  a $K3$ surface:
Indeed, $\sigma_2$
gives three sections of the bundle $\bigwedge^2\mathcal{E}_3$ on the
Grassmannian $G(3,V_7)$, where as usual $\mathcal{E}_3$ denotes the dual of the tautological subbundle on the Grassmannian $G(3,7)$. On the vanishing locus of these three sections (that we can
also see via the projection $V_{10}\rightarrow V_7$ as embedded in $G(6,V_{10})$), the section $\sigma$ gives  a section of $\bigwedge^3\mathcal{E}_3$.
Hence $S$ is defined in $G(3,V_7)$ by three sections of $\bigwedge^2\mathcal{E}_3$ and one section
of $\bigwedge^3\mathcal{E}_3$. Thus it has trivial canonical bundle and is in fact
the general member of  the complete family  of $K3$ surfaces of genus $12$ described by Mukai
\cite{mukai}.

\subsection{O'Grady 10-dimensional examples and intermediate jacobian fibrations}\label{subseclsv}
This section is devoted to the hyper-K\"ahler manifolds $\mathcal{J}$ constructed
in \cite{lsv} as a $K$-trivial compactification of the intermediate Jacobian
fibration $\mathcal{J}_U\rightarrow U$ associated to the universal family
$\mathcal{Y}_U\rightarrow U$ of smooth hyperplane sections of a general cubic fourfold
$X\subset \mathbb{P}^5$. Here $U\subset (\mathbb{P}^5)^*$ is the Zariski  open set parameterizing smooth
hyperplane sections of $X$.
Our aim is to give a new proof of \cite[Corollary 6.2]{lsv}:
\begin{theo} \label{theoog10} The varieties $\mathcal{J}$ are deformations of O'Grady's 10-dimensional
hyper-K\"ahler
manifolds.
\end{theo}
The original proof was obtained by specializing $X$ to a general Pfaffian
cubic fourfold $X_{\textrm{Pf}}$. The proof that $\mathcal{J}_{X_{\textrm{Pf}}}$ exists and is smooth does not necessitate   much extra work  but the
proof that it is birational to the O'Grady moduli space $M_{4,2,0}(S)$ (where $S$ is the associated
$K3$ surface of degree $14$ as in Section \ref{subsecbedo}) is rather involved and uses work of Markushevich--Tikhomirov
 \cite{marku} and Kuznetsov
\cite{kuznetsov} on Pfaffian geometry in the threefold case.
We are going to use here a different degeneration which was introduced by Hassett
\cite{hassett}, and plays an important role in \cite{laza}, \cite{looijenga}.
Let $X_0$ be the chordal cubic fourfold which is defined as the secant variety
of the Veronese surface $V\subset \mathbb{P}^5$.
Blowing up the parameter point $[X_0]$ in the space of all cubics, the general point
of the exceptional divisor determines a cubic $X_\infty$, (or rather its restriction to
$X_0$). The restriction of $X_\infty$ to $V$ gives a sextic curve $C\subset\mathbb{P}^2\cong V$, hence a $K3$ surface obtained as the double cover
$r:S\rightarrow \mathbb{P}^2$ of $\mathbb{P}^2$ ramified along $C$. It is proved in
\cite{laza}, \cite{looijenga} that
the period map defined on the regular part of the pencil
$\langle X_0,X_\infty\rangle$ extends over $0$ (in particular the monodromy
on degree $4$ cohomology of the smooth fibers $X_t$ is finite) and the
limit Hodge structure is that of $H^2(S)$.

A hyperplane section $Y_0=H\cap X_0$ of $X_0$ determines by restriction to $V$ a conic
$C=H\cap V$ in $\mathbb{P}^2$ whose inverse image $C'=r^{-1}(C)$ is a hyperelliptic curve of genus
five. The degeneration of a smooth cubic threefold $Y_t=H\cap X_t$ to
a pair $(Y_0,Y_\infty)$, consisting of the Segre cubic threefold (secant variety of a normal quartic curve $\mathbb{P}^1\cong C_0\subset\mathbb{P}^4$) and a cubic hypersurface section $Y_\infty=X_\infty\cap Y_0$ of it, is studied first in \cite{collino}, see also \cite{allcocktoledo}. It is proved there that
the intermediate Jacobian $J(Y_t)$ specializes to
the Jacobian $J(C')$, where $C'$ is the hyperelliptic curve defined as the double cover
of $C_0$ ramified at the $12$ points of $C_0\cap Y_\infty$.
\begin{rema}\label{remamult}{\rm  Note  that
 under a general one-parameter degeneration of a cubic threefold to the Segre cubic threefold,
 the hyperelliptic Jacobian over $0$  is a smooth   (in particular} reduced {\rm ) fiber of the associated  one-parameter family of intermediate Jacobians. This is clear since we are actually working with abelian varieties and not torsors (there is a $0$-section).}
 \end{rema}

 Coming back to the associated $K3$ surface $r:S\rightarrow \mathbb{P}^2$,  ramified along a sextic curve, the Veronese surface $V=v(\mathbb{P}^2)$ is contained in
$\mathbb{P}^5$ and the projective space $(\mathbb{P}^5)^*$ parameterizes the universal
family $\mathcal{C}\rightarrow (\mathbb{P}^5)^*$ of conics in $\mathbb{P}^2$ and
the universal family
$\mathcal{C}'\rightarrow (\mathbb{P}^5)^*$ of hyperelliptic curves $r^{-1}(\mathcal{C}_t)$
on $S$.
It follows from this discussion that
if $\mathcal{X}\rightarrow B$ is a general one-parameter family of cubic fourfolds with central fiber
$X_0$ and fist order deformation determined by a generic $X_\infty$, then the corresponding
family $\mathcal{J}_{\mathcal{X}}$ (which is well defined
over a Zariski open set of $B$ and is a family
of projective hyper-K\"ahler varieties)  has a component of its central fiber which is
birational to the Jacobian fibration $\mathcal{J}_{\mathcal{C}'}$.

The following fact already appears in \cite{ogradyrapagnetta}:
\begin{prop} \label{proograp} Let $r:S\rightarrow \mathbb{P}^2$ be a $K3$ surface as above. Assume ${\rm Pic}\,S=\mathbb{Z}$.
 Then   the Jacobian fibration
$\mathcal{J}_{\mathcal{C}'}\rightarrow (\mathbb{P}^5)^*$ of  the universal family of curves
$\mathcal{C}'\rightarrow (\mathbb{P}^5)^*$ is birational to the O'Grady moduli space
$M_{4,2,0}(S)$ of  rank $2$ vector bundles on $S$, with trivial determinant and $c_2=4$.
\end{prop}
\begin{proof} Denoting $H=r^*\mathcal{O}(1)\in {\rm Pic}\,S$, the curves
$C'$ belong to the linear system $|2H|$ on $S$. Let $E$ be a general stable  rank $2$ vector bundle on $S$ with $c_2=4$ and $c_1=0$.
One has $\chi(S,E(H))=2$ and  $H^1(S,E(H))=0=H^2(S,E(H))$ as shows specialization to the case of the torsion free
sheaf $I_z\oplus I_{z'}$ where $z$ and $z'$ are two general subschemes of length $2$ on $S$.
Thus $E$ has two sections and is generically generated by them, again by the same
specialization argument.
So we have an injective evaluation map
$W\otimes \mathcal{O}_S\rightarrow E(H)$, and its determinant vanishes along  a
curve $C'\in |2H|$. The cokernel of the evaluation map is then a line bundle $L$ of
degree $2$ on $C'$, as it has $H^0(C,L')=0,\,H^1(C,L')\cong W$.
Conversely, start with a general curve $C'\in |2H|$ and a general line bundle
$L'$ of degree $2$ on $C'$. Then $H^0(C',K_{C'}-L')$ has dimension $2$, and the Lazarsfeld--Mukai bundle associated to the pair $(C',L')$ provides a rank $2$ bundle with the desired Chern classes
on $S$.
Thus we constructed a birational map between $M_{4,2,0}(S)$ and the relative
Picard variety $\mathcal{J}_{\mathcal{C}',2}$ of line bundles of degree $2$ on the family
$\mathcal{C}$ of curves
$C'$, which is in fact birational to $\mathcal{J}_{\mathcal{C}'}$ since the curves
${C}'$ are hyperelliptic. Indeed, the hyperelliptic divisor
gives a section of $\mathcal{J}_{\mathcal{C}',2}$ which provides the isomorphism above by translation.
\end{proof}
 Theorem \ref{theoog10} now follows from Proposition \ref{proograp} and Theorem \ref{theo1}. The only
 thing to check is the fact that under  a general one-parameter degeneration of a cubic fourfold
 to the secant variety $X_0$ to the Veronese surface in $\bP^5$, the hyperelliptic Jacobian fibration $\mathcal{J}_{\mathcal{C}'}$ introduced above appears
 as a multiplicity $1$ component  in the central fiber of the
 associated family of intermediate Jacobian fibrations. As these varieties are fibered over
 a Zariski open set of $(\mathbb{P}^5)^*$, the fact that this component
  has multiplicity $1$ follows  from Remark \ref{remamult}.  The proof is thus complete.

\subsection{LLSvS eightfolds}
The LLSvS eightfolds were constructed in \cite{llsvs}, and were proved in \cite{addingtonlehn} (see also \cite{clehn})
to be deformation equivalent to $S^{[4]}$. These hyper-K\"ahler manifolds are constructed as follows:
Start from a general cubic fourfold $X$ and consider the Hilbert scheme
$\mathcal{H}_3$ of degree $3$ rational curves in $X$. Then
$\mathcal{H}_3$ is birational to a $\mathbb{P}^2$-bundle over a hyper-K\"ahler manifold
$Z(X)$.
The following is proved in \cite{addingtonlehn}:
\begin{theo} If $X\subset\mathbb{P}(\bigwedge^2V_6)$ is Pfaffian, then
$Z(X)$ is birational to $S^{[4]}$, where $S\subset G(2,V_6)$ is the associated $K3$ surface as in Section
\ref{subsecbedo}.
\end{theo}
This result, combined with Huybrechts' Theorem \ref{theohuy}, implies:
\begin{coro} The varieties $Z(X)$ are deformation equivalent to $S^{[4]}$.
\end{coro}
Let us now give another proof of this last result, based on Theorem \ref{theo1} and
the degeneration to the chordal cubic.
In \cite{lsv}, it is noticed that
the varieties $\mathcal{J}(X)$ and  $Z(X)$ are related as follows:
\begin{lem}\label{le2juin} The  relative  Theta divisor of
the intermediate Jacobian fibration $\mathcal{J}_U$ of $X$ (which is canonically defined)   is birationally a $\mathbb{P}^1$-bundle over $Z(X)$.
\end{lem}
\begin{proof} Indeed, we know by Clemens--Griffiths \cite{CG} that the Theta divisor in the intermediate Jacobian
of a cubic threefold $Y$ is  parameterized via the Abel-Jacobi map of $Y$
by degree $3$ rational curves on $Y$, the fiber passing through a general
 curve $[C]\in \mathcal{H}_3(Y)$ being the $\mathbb{P}^2$ of deformations of $C$ in
 the unique cubic surface $\langle C\rangle \cap X$ containing $C$.
 It follows from this result  that the
  relative Theta divisor $\Theta\subset \mathcal{J}_U$ parameterizes the data of such a $\mathbb{P}^2_C\subset \mathcal{H}_3(X)$ and of a hyperplane section $Y$ of $X$ containing the cubic surface
 $\langle C\rangle$. This is clearly birationally a $\mathbb{P}^1$-bundle over $Z(X)$.
 \end{proof}
 We now specialize $X$ to the chordal cubic $X_0$, or more precisely
 to a point of the exceptional divisor of the blow-up of this point in
 the space of all cubics, which determines as in the previous section
 a degree $2$ $K3$ surface $r:S\rightarrow \mathbb{P}^2=V$.
 We use the fact already exploited in the previous section that
 the intermediate Jacobian fibration $\mathcal{J}_U$ then specializes
 birationally to the Jacobian fibration $\mathcal{J}_{\mathcal{C}'}$ associated to the family
$\mathcal{C}'$  of hyperelliptic curves $C'=r^{-1}(C)$, $C$ being a conic in $\mathbb{P}^2$.
 The Theta divisor $\Theta\subset \mathcal{J}_U$ specializes to the Theta divisor
 $\Theta_{\mathcal{C}'}$ which is indeed contained in $\mathcal{J}_{\mathcal{C}'}$ since
 the curves $C'$ have a natural degree $4$ divisor (the canonical Theta divisor is naturally contained
 in ${\rm Pic}^4(C')$ for a genus $5$ curve $C'$, so by translation using $H_{\mid C'}$, we get
 it contained in ${\rm Pic}^0(C')$).
 We now have:
\begin{prop} \label{proS4} The divisor $\Theta_{\mathcal{C}'}\subset \mathcal{J}_{\mathcal{C}'}$ is birational to
a $\mathbb{P}^1$-bundle over $S^{[4]}$.
\end{prop}
\begin{proof} Let us identify $\mathcal{J}_{\mathcal{C}'}$ to $\mathcal{J}^4_{\mathcal{C}'}$ via translation by the section $[C']\mapsto H_{\mid C'}$ of $\mathcal{J}^4_{\mathcal{C}'}$. Then $\Theta_{\mathcal{C}'}\subset \mathcal{J}^4_{\mathcal{C}'}$ is the family of effective divisors of degree $4$ in curves $C'\subset S$.
Such an effective divisor determines a subscheme of length $4$ in $S$. This gives a rational map
$\phi:\Theta_{\mathcal{C}'}\dashrightarrow S^{[4]}$.
Given a generic  subscheme $z\subset S$ of length four, $z$ is contained in a pencil of curves $C'\subset S$ and determines an effective divisor of degree $4$ in each of them, showing
that the general  fiber of $\phi$ is a $\mathbb{P}^1$. This shows that, via $\phi$,
$\Theta_{\mathcal{C}'}$ is birationally  a $\mathbb{P}^1$-bundle over $S^{[4]}$.
\end{proof}
As a consequence of Lemma \ref{le2juin} and Proposition \ref{proS4}, we conclude that in the given degeneration, the
central fiber of the family $\mathcal{Z}$ of LLSvS eightfolds has a component which is
birational to $S^{[4]}$, so that (leaving to the reader to check the multiplicity $1$ statement for the considered component of the central fiber),
we can apply Theorem \ref{theo1} and conclude that
$\mathcal{Z}_s$ is deformation equivalent to $S^{[4]}$.

%%%%%%%%%%%%%%%%%%%%%%%%%%%%%%%%%%%%%%%%
%%% S6
\section{The dual complexes  for degenerations of hyper-K\"ahler manifolds}\label{secdual}
While most of the paper is concerned with the case of finite monodromy degenerations, we close here by making some remarks on the infinite monodromy case. We start by recalling the case of $K3$ surfaces. Namely, the Kulikov--Persson--Pinkham Theorem (\cite{kulikov,pp}) states that any projective $1$-parameter degeneration $\calX/\Delta$ of $K3$ surfaces can be arranged   to be semistable with trivial canonical bundle; such a degeneration is  called a {\it Kulikov degeneration} of $K3$s.
For a Kulikov degeneration, one can give a rather precise description of the  possible central fibers $X_0$ of the degeneration (depending on the Type as defined in \ref{defType}).

\begin{theo}[{Kulikov \cite[Theorem II]{kulikov},  Persson \cite{persson}, Roan \cite{roan}}]\label{theokulikov2}
Let $\calX/\Delta$ be a Kulikov degeneration of $K3$ surfaces. Then, depending on the Type of the degeneration (or equivalently, the nilpotency index of $N$) the central fiber $X_0$ of the degeneration is as follows:
\begin{itemize}
\item[i)] {\it Type I:} $X_0$ is a smooth $K3$ surface.
\item[ii)] {\it Type II:} $X_0$ is a chain of surfaces, glued along smooth elliptic curves. The end surfaces are rational surfaces,  and the corresponding double curves are smooth anticanonical divisors. The intermediary surfaces (possibly none) are (birationally) elliptically ruled; the double curves for such surfaces are two distinct sections which sum up to an anticanonical divisor.
\item[iii)] {\it Type III:} $X_0$ is a normal crossing union of rational surfaces such that the associated dual complex is a triangulation of $S^2$. On each irreducible component $V$ of $X_0$, the double curves form a cycle of rational curves giving an anticanonical divisor of $V$.
\end{itemize}
\end{theo}

\begin{rema}\label{remdualcx} As usual, we let $\Sigma$ be the dual complex associated to the normal crossing variety $X_0$, the central fiber of the Kulikov degeneration.  Then, the topological realization $|\Sigma|$ is either a point, an interval, or $S^2$ according to the Type (I, II, III) of the degeneration. In particular, $\dim|\Sigma|=\nu -1$, where $\nu$ is the nilpotency index of $N$.
\end{rema}
The purpose of this section is to give partial generalizations of Kulikov classification of the central fiber in a degeneration of hyper-K\"ahler manifolds (and make some remarks on the general $K$-trivial case). To start, as already noted, Theorem \ref{theo3} is nothing but a strong generalization of Kulikov's Theorems in the Type I case (see Def. \ref{defType}). Informally, {\it a finite monodromy degeneration of hyper-K\"ahler manifolds admits a smooth filling}. The focus in this section is on the Type II and III cases. Namely, we will discuss some generalization of Remark \ref{remdualcx} to the higher dimensional case and a partial resolution of a conjecture of Nagai \cite{nagai}  concerning the monodromy action on higher cohomology groups.

In contrast to the case of $K3$ surfaces, for higher dimensional hyper-K\"ahler manifolds, the cohomology in higher degree (than $2$) is non-trivial, and thus a natural first question is to what extent the nilpotency index for the monodromy on this higher cohomology is determined by the Type (or equivalently the nilpotency index on $H^2$). This question was investigated by Nagai \cite{nagai} who obtained specific results in the case of degenerations of Hilbert schemes of $K3$s and Kummer case, and made the following natural conjecture:
\begin{conjecture}[{Nagai \cite[Conjecture 5.1]{nagai}}]\label{conjnagai} For a degeneration of hyper-K\"ahler
$$\nilp(N_{2k})=k(\nilp(N_2)-1)+1.$$
(i.e. the nilpotency order on $H^{2k}$ is determined by that on $H^2$).
\end{conjecture}
\begin{rema}
There is difference of $1$ between our nilpotency index, and that used by Nagai: for us $N$ has index $\nu$ if $\nu$ is minimal such that $N^{\nu}=0$, while in \cite{nagai}, $N$ has index $\nu$ if $N^{\nu+1}=0$ (and $N^\nu\neq 0$).
\end{rema}

The main result of Nagai (\cite[Thm. 2.7, Thm. 3.6]{nagai}) is that the conjecture is true for degenerations arising from Hilbert schemes of $K3$s or generalized Kummers associated to families of abelian surfaces. Below, we check the conjecture in the Type I and III cases
(see Corollary \ref{coronagaiIII}). Furthermore, we get some results on the topological type of the dual complex of the degeneration (see Thm. \ref{theodualcx} and Thm. \ref{theostructureskeleton}).

\begin{theo}\label{theonagai}
Nagai's Conjecture holds in Type I and III cases.  For Type II, it  holds  $\nilp(N_{2k})\in\{k+1,\dots,2k-1\}$ for $k\in\{2,\dots,n-1\}$.
\end{theo}
\begin{rema}  The case of Type I is Corollary \ref{corofinitemonoevrywhere}, a consequence of Theorem \ref{theo3}.
\end{rema}

\subsection{Essential skeleton of a $K$-trivial degeneration}\label{secskeleton} Let $\calX/\Delta$ be a semistable degeneration of algebraic varieties. An important gadget associated to the degeneration is the dual complex $\Sigma$ of the normal crossing variety $X_0$ (the central fiber of the degeneration). The dual complex encodes the combinatorial part of the degeneration and can be used to compute the $0$-weight piece (which reflects the combinatorial part) of the MHS on $X_0$ and of the LMHS. Specifically, an easy consequence of the Clemens--Schmid exact sequence (see \cite{morrison}, \cite{arapura}) gives:
\begin{equation}\label{eqwt0}
H^k(|\Sigma|)\cong W_0 H^k(X_0) \cong W_0H^k_{\lim}.
\end{equation}
The first identification is almost tautological; it follows from the Mayer--Vietoris spectral sequence computing the cohomology of $X_0$. While the second follows from a weight analysis of the Clemens--Schmid sequence, which (in particular) shows that the natural specialization map $H^k(X_0)\to H^k_{\lim}$ has to be an isomorphism for weight $0$. We note that there is a much more general version of the second identity. Namely, as explained in Section \ref{secmono}, as a consequence of \cite{kk} and \cite{steenbrink}, as long as $X_0$ is semi log canonical (e.g. normal crossing), the specialization map $H^k(X_0)\to H^k_{\lim}$ is an isomorphism on the $I^{p,q}$ pieces with $p\cdot q=0$. In particular, we get an isomorphism for the weight $0$ pieces (corresponding to $p=q=0$) of the MHS on $X_0$ and the LMHS.

The semistable models are not unique, and thus the topological space $|\Sigma|$ depends on the model (e.g. $|\Sigma|$ might be a point, but  after a blow-up  might become an interval). In order to obtain a more canonical topological space one needs to require some ``minimality'' for the semistable model.  While many ideas towards an intrinsic definition for $|\Sigma|$ occur in the literature (e.g. Kulikov's results can be regarded as the starting point), the right definitions  were only recently identified by de Fernex--Koll\'ar--Xu  \cite{dkx}. Namely, the minimality corresponds to a relative minimal model in the sense of MMP. This, however leads to singularities for $\calX/\Delta$ and the central fiber $X_0$. It turns out that the right class of singularities that still allow the definition of a meaningful dual complex  is dlt. In other words, the correct context for defining an intrinsic dual complex associated to a degeneration is that of minimal dlt degeneration (see Appendix \ref{dltpairs}). The minimal dlt model $\calX/\Delta$ is not unique, but changing the model has no effect on $|\Sigma|$ (the associated topological spaces will be related by a PL homeomorphism, see \cite[Prop. 11]{dkx}). On the other hand, if $\calX'/\Delta$ is a semistable resolution of $\calX/\Delta$, then the topological realization $|\Sigma|$ associated to the canonical dual complex is a deformation retract of the topological realization of the dual complex associated to the semistable resolution $\calX'/\Delta$, and thus the two topological spaces are homotopy equivalent.
\begin{rema}
Let us note that the semi-log-canonical (slc) singularities are too degenerate to lead to a good notion of dual complex. For instance, it is easy to produce KSBA degenerations of $K3$ surfaces of Type III such that the central fiber $X_0$ is a normal surface with a single cusp singularity (e.g. such examples occur in the GIT analysis for quartic surfaces, see \cite{shah4}). The (naive) dual complex in this situation would be just a point, while from KPP Theorem, the intrinsic dual complex is in fact $S^2$.
\end{rema}

In the case of $K$-trivial degenerations, there is an alternative approach (yet producing the same outcome) coming from mirror symmetry in the Kontsevich--Soibelman interpretation.  This in carefully worked out in Mustata--Nicaise \cite{mn} (via Berkovich analytification). Relevant for us is the fact that associated to a $K$-trivial degeneration $\calX/\Delta$ there is an intrinsic (depending only on $\calX^*/\Delta^*$) topological space, that we call (following \cite{mn}) {\it the essential skeleton}, $\Sk(\calX)$ associated to the degeneration. For a minimal dlt degeneration of $K$-trivial varieties, the essential skeleton $\Sk(\calX)$ can be identified with the topological realization $|\Sigma|$ of the dual complex (cf. Nicaise--Xu \cite[Thm. 3.3.3]{nx}). (As discussed in Section \ref{secmmp} and \cite[Theorem 1.1]{fujino-ss}, a minimal dlt model always exists. Two such models are birationally crepant, leading to $\Sk(\calX)$ being well defined.)  Finally, Nicaise--Xu \cite[Thm. 3.3.3]{nx} show that $\Sk(\calX)$ is a pseudo-manifold with boundary.

The purpose of this section is to make some remarks on the structure of the essential skeleton $\Sk(\calX)$ for a degeneration of hyper-K\"ahler manifolds (depending on the Type of the degeneration). We note that there is an extensive literature on the related case of Calabi-Yau varieties (esp. relevant here is Koll\'ar--Xu \cite{kx}), and that several papers (esp. \cite{mn}, \cite{nx}, \cite{kx}) treat the general $K$-trivial case. However, to our knowledge, none of the existing literature discusses the skeleton $\Sk(\calX)$ in terms of the Type (I, II, III) of the hyper-K\"ahler degeneration.

\begin{rema}
Recently, Gulbrandsen--Halle--Hulek \cite{hulek} (see also \cite{hulek2} and \cite{nagai2}) have studied explicit models for certain types of degenerations of Hilbert schemes of surfaces.
In particular, starting with a Type II degeneration of $K3$ surfaces $\mathcal S/\Delta$,   it is constructed in \cite{hulek2} an explicit  minimal dlt degeneration  for the associated Type II family of Hilbert schemes $\calX/\Delta$ of $n$-points on $K3$ surfaces (with $X_t=(S_t)^{[n]}$).
From our perspective here, most relevant is the fact that the $\Sk(\calX)$ is the $n$-simplex. For comparison, our results (see Theorem \ref{theodualcx}) will only say $\dim \Sk(\calX)=n$ and that $\Sk(\calX)$ has trivial rational cohomology.
\end{rema}

\subsection{Type III is equivalent to the MUM case}  Considering as above a one-parameter degeneration $f: \calX\rightarrow \Delta$, we assume additionally that $f$ is projective. It is then well-known that the monodromy
$\gamma_k$ acting on $H^k(X_t,\mathbb{Q}),\,t\in \Delta^*$ is quasi-unipotent, that is $(\gamma_k^N-Id)^m=0$
 for some integers $N,\,m$. Furthermore
 one can take $m\leq k+1$.
\begin{definition} We will say that the monodromy on $H^k$ is {\it maximally unipotent}
 if the minimal order $m$ is $k+1$.
Let $\calX/\Delta$ be a degeneration of $K$-trivial varieties of dimension $n$. We say that {\it the degeneration is maximally unipotent (or MUM)} if the nilpotency index for the monodromy action on $H^n(X_t,\bQ)$ is $n+1$ (the maximal possible index).
\end{definition}

It is immediate to see that in a MUM degeneration, the skeleton has dimension at least $n$. For $K$-trivial varieties, a strong converse also holds:

\begin{theo}[{Nicaise-Xu \cite{nx}}]
Let $\calX/\Delta$ be a degeneration of $K$-trivial varieties of dimension $n$.
\begin{itemize}
\item[i)] If the degeneration is MUM, then $\Sk(\calX)$ is a pseudo-manifold of dimension $n$.
\item[ii)] Conversely, if $\Sk(\calX)$ is of dimension $n$, the degeneration is MUM.
\end{itemize}
\end{theo}
\begin{rema}
We note here that both the minimality and $K$-triviality are essential conditions (see also Rem. \ref{remktirvmin}). Dropping the $K$-triviality, we can consider a family of genus $g\ge 2$ curves degenerating to a compact type curve. Then the monodromy is finite, but the dual graph of the central fiber is an interval. Similarly, one can start with a family of elliptic curves and blow-up a point. This will give a non-minimal family, with trivial monodromy, and dual graph of the central fiber an interval.
\end{rema}

We note that one additional topological constraint on the  skeleton $\Sk(\calX)$ is that it is  simply connected.
\begin{prop}
Let $\calX/\Delta$ be a degeneration such that $\pi_1(X_t)=1$. Then $\pi_1(\Sk(\calX))=1$.
\end{prop}
\begin{proof} \cite[\S34 on p. 541]{kx}.
\end{proof}
Mirror symmetry makes some predictions on the structure of essential skeleton $\Sk(\calX)$ for MUM degenerations. Briefly, the situation is as follows:
\subsubsection{}\label{subssyz}The SYZ conjecture (\cite{syz}) predicts the existence of a special Lagrangian fibration $X/B$ for $K$-trivial varieties near the large complex limit point (the cusp of the moduli corresponding to the MUM degeneration). Furthermore, SYZ predicts that the mirror variety is obtained by dualizing this Lagrangian fibration.
\subsubsection{}Kontsevich--Soibelman \cite{ks2,ks} predict that the base $B$ of the Lagrangian fibration is homeomorphic to the essential skeleton $\Sk(\calX)$. In fact, $B$ is predicted to be the Gromov--Hausdorff limit associated to $(X_t,g_t)$ where $g_t$ is an appropriately scaled Ricci-flat Yau metric on the (polarized) smooth fibers $X_t$. This gives a much richer structure to $B$ ({\it Monge-Ampere manifold}, see \cite[Def. 6]{ks}). The underlying topological space is expected to be $\Sk(\calX)$ (e.g. \cite[\S6.6]{ks}). As already mentioned, the Kontsevich--Soibelman predictions led to the Mustata--Nicaise \cite{mn} definition of $\Sk(\calX)$.

\subsubsection{}The case of $K3$ surfaces is quite well understood (see \cite{ks}, \cite{gw}). In higher dimensions, there is a vast literature on the case of (strict) Calabi-Yau's, most notably the Gross--Siebert program (e.g. \cite{gs}).  From our perspective, we note that the $\Sk(\calX)$ for a MUM degeneration of Calabi-Yau $n$-folds is predicted to be the sphere $S^n$. This is true in dimension $2$ by Kulikov's Theorem, and in dimensions $3$ (unconditional) and $4$ (assuming additionally that the degeneration is normal crossings) by Koll\'ar--Xu \cite{kx}.
\subsubsection{} The SYZ conjecture and the Kontsevich--Soibelman picture for hyper-K\"ahlers is similar to the $K3$ case (see especially Gross--Wilson \cite{gw}). Conjecturally, the special Lagrangian fibration $X/B$ near the large complex limit point can be constructed via a hyper-K\"ahler rotation. Briefly, let $[\Omega]\in H^2(X,\bC)$ and $[\omega]\in H^2(X,\bR)$ be the classes of the holomorphic form and of the polarization (a K\"ahler class) on $X$. The MUM condition implies the existence of a vanishing cycle $\gamma\in H^2(X,\bQ)$ with $q(\gamma)=0$ (where $q$ is the Beauville-Bogomolov form on $H^2$). The problem is that $\gamma$ is not an algebraic class. Recall that, given a hyper-K\"ahler manifold with a fixed K\"ahler class on it,  the space of complex structures on the  hyper-K\"ahler manifold contains a
distinguished $S^2$ (so called twistor family). Using this, one can modify the complex structure on $X$ (call the resulting complex manifold $X'$) such that $\gamma$ is an algebraic class with $q(\gamma)=0$ (essentially, after an appropriate $\bC^*$-scaling of $\Omega$, we can arrange $\Omega'=\textrm{Im}(\Omega)+i\omega$ and $\omega'=\textrm{Re}(\Omega)$ to be the holomorphic and respectively K\"ahler classes on $X'$, and  $\gamma$ to be orthogonal to $\Omega'$). The so-called hyper-K\"ahler SYZ conjecture (which is known in various cases) then predicts  that  (a multiple of) $\gamma$ is  the class of a (holomorphic) Lagrangian fibration $X'/B$. Of course, in the $C^\infty$ category, $X'/B$ is the same as the desired special Lagrangian $X/B$. (From a slightly different perspective, mirror symmetry for hyper-K\"ahler manifolds was studied by Verbitsky \cite{vermirror}.)
\subsubsection{}\label{subshwang} Finally, the basis of an (algebraic) Lagrangian fibration $X'/B$ is expected to be $\bC\bP^n$ (for $2n$-dimensional hyper-K\"ahler manifolds). For instance, if $B$ is smooth, then $B\cong \bC\bP^n$ by a theorem of Hwang \cite{hwang}.

\subsubsection{} To conclude, mirror symmetry (via SYZ conjecture and Kontsevich--Soibelman) predicts that {\it the essential skeleton $\Sk(\calX)$ for a MUM degeneration is $S^n$ and respectively $\bC\bP^n$ for Calabi-Yau's and respectively hyper-K\"ahler's}. The following result is a weaker version of this statement, saying that it holds in a cohomological sense.

If $X$ is a simply connected compact K\"ahler manifold with trivial
canonical bundle, the Beauville-Bogomolov decomposition theorem
\cite{beau} says that
$X\cong \prod_i X_i$ where the $X_i$ are either Calabi-Yau of dimension $k_i$ (that is with
$\SU(k_i)$ holonomy group), or irreducible hyper-K\"ahler of dimension
$2l_j$ (that is with
$\Sp(2l_j)$ holonomy group). The type of the decomposition will be  the collection of the dimensions
$k_i$, $2l_j$ (with their multiplicities).

\begin{theo}\label{theostructureskeleton}
Let $\calX/\Delta$ be a minimal dlt degeneration of $K$-trivial varieties. Assume that the general fiber $X_t$ is a simply connected $K$-trivial variety.

Assume that the degeneration is maximal unipotent. Then
\begin{itemize}
\item[i)] $H^*(\Sk(\calX),\bQ)\cong \prod_i H^*(S^{k_i},\bQ)\times \prod_j H^*(\bC\bP^{l_j},\bQ))$, where $k_i$ represent the dimensions of the Calabi-Yau factors and $2l_j$ the dimensions of the hyper-K\"ahler factors in the Beauville-Bogomolov decomposition of the general fiber $X_t$.
\item[ii)] Conversely, the cohomology algebra of the skeleton $\Sk(\calX)$ determines the  type  of the Beauville-Bogomolov decomposition of $X_t$.
\end{itemize}
\end{theo}

\begin{proof}[Proof of Theorem \ref{theostructureskeleton}]
Let $\calX/\Delta$ be a minimal dlt degeneration. By the du Bois arguments of Section \ref{secmono}, the weight $0$ pieces of the limit mixed Hodge structure on $H^*_{\lim}$ are identified with the weight $0$ pieces of the mixed Hodge structure on $H^*(X_0)$. Next, a Mayer--Vietoris argument identifies $W_0H^k(X_0)$ with $H^k(\Sk(\calX))$ (recall $\Sk(\calX)$ is nothing but the topological realization of the dual complex in this situation). In other words, we see that \eqref{eqwt0} holds in the situation of minimal dlt degenerations. Summing over all degrees $k$ gives an algebra structure, and then an identification of  the algebra associated to the weight $0$ piece of the LMHS with the cohomology algebra of $\Sk(\calX)$. Here it is important to
note that this identification is not only as vector spaces, but rather as algebras (i.e. compatible with the cup product) -- this is discussed in  Lemma \ref{lemmacup} below.

It remains to understand the algebra structure for the weight $0$ piece of the LMHS (under the MUM assumption). It is immediate to see that on $H^k_{\lim}$ the weight $0$ piece is non-zero if and only if the monodromy action on $H^k(X_t)$ is maximally unipotent. When this is satisfied, we have $N^k:\Gr_{2k}^W H^k_{\lim}\cong W_0H^k_{\lim}$, and then $\Gr_{2k}^W H^k_{\lim}\subset F^k H^k_{\lim}\cong H^{k,0} (X_t)$ (as vector spaces). Thus, the weight $0$ piece can be identified with a subspace in the space of degree $k$ holomorphic forms on $X_t$. The following proposition
tells us that under the MUM assumption on the top degree cohomology,
the weight $0$ piece can be identified with  the whole space of degree $k$ holomorphic forms on $X_t$.
\begin{prop} \label{lemum} Let $\calX/\Delta$ be a projective degeneration of $K$-trivial varieties.
 \begin{itemize}
\item[i)] Assume the fibers $X_t$ are simply connected  Calabi-Yau manifolds (so $h^i(X_t,\mathcal{O}_{X_t})=0$ for $0<i<n={\rm dim}\,X_t$). Then
 the only degree in which the monodromy can be maximally unipotent is $n$.
\item[ii)] Assume the fibers $X_t$ are hyper-K\"ahler manifolds  (so $h^i(X_t,\mathcal{O}_{X_t})=0$ for $i$ odd and $\mathbb{C}$ for $i=2j$, $0<i<2n={\rm dim}\,X_t$). Then the only degrees where the monodromy
 can be maximally unipotent are the even degrees $2i$ and
  the monodromy is  maximally unipotent in some degree $k=2i$ if and only if it is maximally unipotent in all degrees $2i\leq 2n$. In particular, MUM degeneration is equivalent to Type III degeneration (for hyper-K\"ahler manifolds).
  \end{itemize}
   \end{prop}
 \begin{proof} (i) As we have $H^{i,0}(X_t)=0$ for $0<i<n$, the Hodge structure on $H^i(X_t,\mathbb{Q})$ has coniveau $\geq 1$. The variation of Hodge structure
 on $R ^if_*\mathbb{Q}$ is thus the Tate
 twist of an effective polarized  variation of Hodge structure of weight $i-2$. Hence its
 quasi-unipotency index is $\leq i-1$.

 (ii) The same argument applies to show that monodromy is not maximally unipotent on cohomology of
 odd degree if $X_t$ is hyper-K\"ahler, since $H^{2i+1,0}(X_t)=0$.
 We know by Verbitsky (Thm. \ref{Verbitsky cohomology}) that
 in degree $2i\leq 2n$, we have an injective map
 given by cup-product
 $$\mu_{i,t}:{\rm Sym}^{i} H^2(X_t,\mathbb{Q})\hookrightarrow  H^{2i}(X_t,\mathbb{Q}),$$
 which more generally induces an injection of local systems on
 $\Delta^*$
 $$\mu_i: {\rm Sym}^{i} (R^2f_*\mathbb{Q})\hookrightarrow R^{2i}f_*\mathbb{Q}.$$
 Note that $\mu_i$ is an morphism of variations of Hodge structures.
  Next,  using a relatively ample line bundle on $f$, we have an orthogonal
 decomposition
 $$R^{2i}f_*\mathbb{Q}={\rm Im}\,\mu_i\oplus B^{2i}$$
 where the local system $B^{2i}$ carries a polarized variation of Hodge structures of weight
 $2i$ with trivial $(2i,0)$-part, as the map
 $\mu_{i,t}$ induces a surjection on $(2i,0)$-forms.
 Applying the same argument as before, we conclude that the monodromy action on $B^{2i}$
 is of quasiunipotency index $\le 2(i-1)+1(<2i+1)$, so the monodromy acting on
 $H^{2i}$ is maximally unipotent if and only if it is maximally unipotent on
 ${\rm Sym}^i H^{2}(X_t,\mathbb{Q})$.
 It is then easy to see that this is the case if and only if it is maximally unipotent
 on $H^2(X_t,\mathbb{Q})$ (e.g. \cite[Lemma 2.4]{nagai}).
 \end{proof}

 By this proposition, the assumption $W_0H^n_{\lim}\neq 0$ implies in fact $W_0H^k_{\lim}\cong H^{k,0}(X_t)$ for all $k$ (since there is a $1$-dimensional contribution for each
 factor of Beauville--Bogomolov decomposition and it has to be maximally unipotent; it is here where we use in an essential
  way the assumption of the theorem that $X_t$ is simply connected, so that we can exclude the abelian variety factors in the Beauville--Bogomolov decomposition).

   These identifications are compatible with the cup product by the following lemma:  \begin{lem}\label{lemmacup} Let $\mc Z \rightarrow \Delta$ be a proper holomorphic map, smooth over
$\Delta^*$, with a central fiber $Z_0$ which is a (global) normal crossing divisor. We assume for simplicity that if  $Z_i, \,i\in I$ are the components of $Z_0$,
for each $J\subset I$,  $Z_J:=\cap_{j\in J}Z_j$ is either empty or connected.  Let
$\Sigma$ be the dual graph of $Z_0$. It has vertices $I$ and one simplex
$J\subset I$ for each non-empty $Z_J$.
The  two natural maps
$$a: H^*(|\Sigma|,\mathbb{Z})\rightarrow H^*(Z_0,\mathbb{Z}),$$
$$b: H^*(Z_0,\mathbb{Z})\rightarrow H^*(Z_t,\mathbb{Z})$$
are compatible with the cup-product.
\end{lem}
\begin{proof} The map $b$ is the specialization map already appearing in Definition
\ref{defist}, and called $sp_*$ there. It  is obtained by observing that $Z_0$ is a deformation retract of
$\mc Z$, hence has the same homotopy type as $\mc Z$. The map $b$ is then the restriction map
$H^*(\mc Z,\mathbb{Z})\rightarrow H^*(Z_t,\mathbb{Z})$ composed with the inverse of the
restriction  isomorphism $H^*(\mc Z,\mathbb{Z})\rightarrow H^*(Z_0,\mathbb{Z})$. Thus it is clearly
compatible with cup-product.

The map $a$ (which can be defined  using Corollary \ref{spectral sequence}
as the composite map
$H^p(|\Sigma|,\mathbb{Z})=E_2^{p,0}=E_\infty^{p,0}\rightarrow H^q(D,\mathbb{Z})$)
can also be constructed as follows: The realization $|\Sigma|$ of
$\Sigma$ is the union over all the faces $J$ of $\Sigma$ of the
simplices $ \Delta_J$, with identifications given by faces: for $J'\subset J$ the simplex
$\Delta_{J'}$ is naturally a face of $\Delta_J$.
Next we have a simplicial topological space $Z_0^\bullet$ associated to
$Z_0$, given by the $Z_J$ and the natural inclusions
$Z_{J'}\subset Z_J$ for each $J\subset J'$.
Let
$r(Z_0^\bullet)$ be the topological space constructed as the union over all
$J\in \Sigma$ of the $Z_J\times \Delta_J$ with gluings given
by the natural maps
$Z_J\times \Delta_{J'}\rightarrow Z_{J'}\times \Delta_J$ for each
inclusion $J'\subset J$.
There are two obvious continuous maps
$$g: r(Z_0^\bullet)\rightarrow Z_0,$$
$$ f: r(Z_0^\bullet)\rightarrow r(\Sigma).$$
The first map is just the projection to $Z_J$ on each $Z_J\times \Delta_J$, followed by the inclusion
in $Z_0$. This map is clearly a homotopy equivalence.
The second map is the  projection to $\Delta_J$ on each $Z_J\times \Delta_J$.
The map $b$ can be defined as the composition of $f^*: H^*(r(\Sigma),\mathbb{Z})\rightarrow
H^*(r(Z_0^\bullet),\mathbb{Z}) $ composed with the inverse of the isomorphism
$g^*:H^*(Z_0,\mathbb{Z})\cong H^*(r(Z_0^\bullet),\mathbb{Z}) $. It follows that $b$ is also compatible with cup-product.
\end{proof}
Together with the previous analysis, we now conclude that in the MUM case,
the cohomology algebra $H^*(|\Sigma|,\mathbb{C})$ is isomorphic to the  algebra of holomorphic forms
$\oplus_iH^0(Z_0,\Omega_{Z_0}^i)$ and also  to the  algebra of holomorphic forms
$\oplus_iH^0(Z_t,\Omega_{Z_t}^i)$.
 We next have  the following lemma.

\begin{lem}\label{lecohalg} Let $X$ be a simply connected compact K\"ahler manifold with trivial
canonical bundle. Then the type of the Beauville-Bogomolov decomposition of $X$ is determined
by the algebra $\oplus_iH^0(X,\Omega_X^i)$.
\end{lem}
\begin{proof} For a Calabi-Yau manifold of dimension $k_i$, there is exactly one holomorphic form
$\omega_i$ of degree $k_i$ and it satisfies $\omega_i^2=0$, while for a hyper-K\"ahler factor $X_j$, the algebra $H^0(X_j,\Omega_{X_j}^\cdot)$ is generated in degree $2$ with one generator
$\sigma_j$ satisfying the equation $\sigma_j^{l_j+1}=0$.
The algebra $A_X^\cdot:=H^0(X,\Omega_X^{\cdot})$ is the tensor product of algebras of these types.
Consider for any integer $k$ the set
$(A_X^2)_k=:\{u\in A_X^2,\,u^{k+1}=0\}$. Let $k_0$ be the
smallest $k$ such that $(A_X^2)_k\not=0$. Then it is immediate that the hyper-K\"ahler
summands are all of dimension $\geq 2k_0$ and that there are exactly
$a_k:={\rm dim}\,(A_X^2)_k$ summands of dimension $2k_0$. The quotient
of $A_X^\cdot$ by the ideal generated by $(A_X^2)_k$ is
the algebra $A_{X'}^\cdot$ of holomorphic forms on the variety $X'$ which is the product
of all Calabi-Yau factors  of $X$ and hyper-K\"ahler factors which are of dimension $>2k_0$.
Continuing with $X'$, we see that the multiplicities of the dimensions of the hyper-K\"ahler factors
are determined by $A_X^\cdot$, and that $A_X^\cdot$ determines the algebra
$A_{X''}^{\cdot}$ of holomorphic forms on the variety $X''$ which is the product of all Calabi-Yau
summands of $X$ of dimension $>2$. It is clear that the latter determines the dimensions
(with multiplicities) of the Calabi-Yau summands of $X$, as they correspond to the degrees
(with multiplicities)
of generators
of the algebra $A_{X''}^{\cdot}$.
\end{proof}

The proof of Theorem \ref{theostructureskeleton} is now complete.
\end{proof}
\begin{coro}\label{coronagaiIII}
Nagai's Conjecture \ref{conjnagai} holds for Type III degenerations of hyper-K\"ahler manifolds.
\end{coro}

\subsection{The Type II case} \label{typeII}
We now focus on the intermediary Type II case. The aim of the subsection is to prove the following result (which together with the results in the Type I and III cases completes the proof of Theorems \ref{theodualcx} and \ref{theonagai}).

\begin{theo}\label{theotype2}
Let $\calX/\Delta$ be a projective degeneration of hyper-K\"ahler manifolds, with $X_t$ smooth of dimension $2n$. Assume that the degeneration has Type II (i.e. $N^2=0$ and $N\neq 0$ on $H^2(X_t)$). Then the following hold:
\begin{itemize}
\item[i)] $\dim \Sk(\calX)=n$;
\item[ii)] For $k\in\{2,n\}$, the index of nilpotency for the monodromy action on $H^{2k}$ is at least $k+1$ and at most $2k-1$.
\end{itemize}
\end{theo}
\begin{proof} Similarly to the Type III case discussed previously, using  \cite[Cor. 2.4]{nagai} and Theorem \ref{Verbitsky cohomology}, we conclude that the nilpotency index  on $\Sym^k H^2$ is $k+1$, and thus the nilpotency index on $H^{2k}$ is at least $k+1$. Conversely, since $H^{2k}/\Sym^k H^2$ is a Hodge structure of level $2k-2$, it follows that the nilpotency index is at most $2k-1$.

To conclude, we note that the arguments of \cite[Claim 32.1]{kx} show that the dimension of $\Sk(\calX)$ is precisely $n$. This is equivalent to saying that the codimension of the deepest stratum in a dlt Type II degeneration is $n$. For a minimal dlt degeneration, we know $X_0$ has trivial canonical bundle. This means that its components are log Calabi-Yau $(V,D)$ with $K_V+D=0$. Inductively, each component of the strata is log Calabi-Yau (e.g. in the $K3$ situation the codimension $1$ components are either elliptic curves or $\bP^1$ with 2 marked points) and is $K$--trivial if and only if it is contained in every component of $X_0$ that intersects it. Hence, a stratum  $ W \subset X_0^{[p]}$  is minimal with respect to inclusion if and only if it has a top holomorphic form (and is thus a $K$--trivial variety with at worst canonical singularities). Moreover, all minimal strata are birational \cite[4.29]{kk-singbook}. It follows that to show that the dual complex has dimension $n$, we only need to produce a top holomorphic form on an $n$--dimensional stratum.
To show this look at the spectral sequence \eqref{ss eq}. We first notice that there is a non zero class in $H^1(\mc O_{X_0^{[1]}})$ which  generates $H^2(\mc O_{X_0})$. To see this we only need to show that there is no contribution from $H^2(\mc O_{X_0^{[0]}})$ and from $H^0(\mc O_{X_0^{[2]}})$. By weights considerations, both statements are clearly true for the spectral sequence of a snc filling. However, since the strata of a dlt filling have rational singularities (Prop. \ref{direct images}) the statement for a snc filling implies that for a dlt filling.
Hence, the only possibility is that a generator for $\bar \eta \in H^2(\mc O_{X_0})$ has to come from a class $\eta \in H^1(\mc O_{X_0^{[1]}})$. By Lemma \ref{cuproduct}, we may consider the product $\eta^n \in H^n(\mc O_{X_0^{[n]}})$ which is non zero since $\bar \eta^n$ has to be non zero and we may conclude that the deepest stratum has dimension $n$.
\end{proof}

\begin{rema} J. Nicaise pointed out that $\dim \Sk(\calX)=n$ follows also
 from Halle-Nicaise \cite{hn} (see Theorem 3.3.3 and esp. (3.3.4) of loc. cit.).
\end{rema}

%%%%%%%%%%%%%%%%%%%%%%%%%%%%%%%%%%%%
\appendix
\section{Reduced dlt pairs} \label{dltpairs}

The purpose of this section is to show that for many aspects reduced dlt pairs behave like snc. Most of the results are well known to the experts (cf. \cite{dkx}).

\begin{definition}\label{defdlt}
A log canonical (lc) pair $(X, D)$ is called dlt if for every divisor $E$ over $(X, D)$ with discrepancy $-1$, the pair $(X,D)$ is snc at the generic point of $center_X(E)$.
\end{definition}

Given a reduced dlt pair $(X, D)$ (i.e.  the divisors appearing in $D= \sum_I D_i$ have coefficient $1$) a stratum of $(X, D)$ is an irreducible component of $D_J:=\cap_J D_i$, for some $J \subset I$.  By \cite[4.16]{kk-singbook}, \cite{fujino} the strata of $(X, D)$ have the expected codimension (i.e. the strata of codimension $k$ in $X$ are the irreducible components of the intersection of $k$ components of $D$) and are precisely the log--canonical centers of $(X, D)$. In particular, $(X, D)$ is snc at the generic point of every stratum and every stratum of codimension $k$ is contained in exactly $k+1$ strata of codimension $k-1$. As noticed in \cite[(8)]{dkx}, this observation is enough to specify the glueing maps needed to define a dual complex. In other words, the dual complex of a dlt pair can be defined just as in the snc case and it satisfies
$$
\Sigma((X,D))=\Sigma((X,D)^{snc}),
$$
where $(X,D)^{snc}$ is the largest open subset of $X$ where the pair $(X, D)$ is snc.
In Proposition \ref{resolution} we show another instance of the fact that ``dlt is almost snc'', namely that given a reduced dlt pair $(X, D)$ we can use the Mayer--Vietoris sequence  \cite{griffiths-schmid} to compute the cohomology of $D$. This was applied in Section \ref{typeII}  to a  minimal dlt degeneration $\calX/D$ of $K$--trivial varieties, since the pair $(\calX,X_0)$ is dlt.

\begin{definition} \cite[(2.78)]{kk-singbook} Let $X$ be a normal variety, let $D \subset X$ be a reduced divisor, and let $f: Y \to X$ be a resolution such that $(Y, D_Y:=f_*^{-1}(D))$ is a snc pair. Then $f: (Y,D_Y) \to (X, D)$ is called rational if
\begin{enumerate}
\item $f_*\mc O_Y(-D_Y)=\mc O_X(-D)$;
\item $R^i{f_*}\mc O_Y(-D_Y)=0$ for $i>0$;
\item $R^if_*\mc \omega_Y(D_Y)=0$ for $i>0$.
\end{enumerate}
\end{definition}

\begin{prop} \label{direct images} Let $(X,D)$ be a reduced dlt pair with $D=\sum_I D_i$ and let $f: (Y,D_Y) \to (X, D)$ be a rational resolution.
For every reduced divisor $D' \le D$, setting $D'_Y:=f^{-1}_* D'$ we have
\be \label{coreq1}
f_*\mc O_Y(-D'_Y)=\mc O_X(-D')\quad \text{ and } \quad
R^i{f_*}\mc O_Y(-D'_Y)=0, \quad \text{ for } i>0.
\ee
 and for every $J \subset I$
\be \label{coreq2}
f_*\mc O_{({D_{Y})}_J}=\mc O_{D_J} \quad \text{ and } \quad R^if_*\mc O_{({D_{Y}})_J}=0 \quad \text{ for } i>0.
\ee

In particular, the induced resolution $f_{|({D_{Y}})_J}:(D_{Y})_J \to D_J$ has connected fibers and every connected component of $D_J$ is irreducible, normal, and has rational singularities.
\end{prop}
\begin{proof} Item \eqref{coreq1} follows from \cite[(2.87),(2.88)]{kk-singbook}. Then
\eqref{coreq2} follows by induction on $|J|$.
\end{proof}

\begin{rema}
From this corollary it follows that in the definition of dual complex of a reduced dlt pair we could consider the connected components of the intersections, rather than the irreducible components (cf. \cite[(8)]{dkx}).
\end{rema}

Let $(X,D)$ be a reduced dlt pair $(X, D)$, with $D= \sum_I D_i$, and fix an ordering of $I$. Denote by $D^{[k]}$  the disjoint union of the strata that have codimension $k$ in $D$.  For a sheaf $\mc F$ on $D$, the Mayer--Vietoris complex of $\mc F$ is
\[
\mc F_{D^\bullet}: \mc F_{D^{[0]}} \to  \mc F_{D^{[1]}} \to \dots \to  \mc F_{D^{[d]}},
\]
where $d=\dim |\Sigma(D)|$, where $\mc F_{D^{[k]}}$ denotes the pullback of $\mc F$ to $D^{[k]}$ via the natural morphism $i_k: D^{[k]} \to D$, and where the differential of the complex is induced by the natural restriction maps $\mc F_{D_J} \to \mc F_{D_{J \cup j}}$, with a plus or a minus sign according to the parity of the position of $j$ in $J \cup j$.

\begin{prop}  \label{resolution}
Let $(X, D)$ be a reduced dlt pair with $D= \sum D_i$. If $\mc F=\mc O_D$ (or is locally free) or $\mc F=\bQ$ (or is a constant sheaf), then $\mathcal{F}_{D^\bullet}$ is a resolution of $\mathcal F$.
\end{prop}
\begin{proof} We start with $\mc F=\mc O_D$. Since $(Y, D_Y)$ is a snc pair, $\mathcal{O}_{D_Y^\bullet}$ is a resolution of $\mathcal O_{D_Y}$ (see for example \cite{fm}).
From Corollary \ref{direct images} it follows both that the complex $f_*\mathcal{O}_{D_Y^\bullet}$ is exact and that $f_*\mathcal{O}_{D_Y^\bullet}=\mathcal{O}_{D^\bullet}$.
Now the case $\mc F=\bQ$.
Let $U \subset X$ be any open set such that on $U \cap X^{snc}$ the divisor $D$ is given by the vanishing of a product of local coordinates. The complex $\Gamma(i^{-1}_\bullet(U \cap X^{snc} ), \bQ_{D^{[\bullet]}})$ is exact except in degree zero, where it has cohomology equal to $\Gamma(U \cap X^{snc} \cap D, \bQ)$ (e.g. \cite{morrison}). As a consequence, the complex is exact on the snc locus.
By the dlt assumption, every connected (hence irreducible) component of $D^{[k]}$ intersects the snc locus of $(X, D)$ so   $\Gamma(i^{-1}_k(U), \bQ_{D^{[k]}})=\Gamma(i^{-1}_k(U \cap X^{snc}), \bQ_{D^{[k]}})$.
Hence, for any $x \in D$ there is a sufficiently small open neighborhood $U$ such that the complex $\Gamma(i^{-1}_\bullet(U), \bQ_{D^{[\bullet]}})  \cong \Gamma(i^{-1}_\bullet(U \cap X^{snc} ), \bQ_{D^{[\bullet]}}) $ is exact except in degree zero where it has cohomology equal to $\Gamma(U\cap D, \bQ)$ and the proposition follows.
\end{proof}

\begin{coro} \label{spectral sequence}
For $(X,D)$ and $\mc F$ as above there is a spectral sequence with $E_1$ term
\be \label{ss eq}
E_1^{p, q}=H^q (\mathcal F_{D^{[p]}})
\ee
abutting to $H^*(\mathcal F)$.
\end{coro}
\begin{proof}
Resolving every term of the complex with its \v Cech complex we get a double complex  which yields a spectral sequence with $E_1$ term equal to \eqref{ss eq}.
\end{proof}

\begin{rema}We notice that Corollary \ref{spectral sequence} for $\mc F=\bC$ implies \eqref{ss eq} for $\mc F=\mc O_D$. Indeed, since the connected components of $ D^{[q]}$ are rational, by  \cite{Kovacs} they are Du Bois and hence there is a surjection
$ H^p ({D^{[q]}}, \bQ) \to \Gr^0_FH^p ({D^{[q]}}, \bQ)=H^p ({D^{[q]}}, \mc O_{D^{[q]}}).
$
By \cite[Thm 2.3.5]{deligneII} $\Gr^k_F$ is an exact functor and hence $ \Gr^0_FH^p ({D^{[q]}})$ abuts to $ \Gr^0_FH^{p+q} (D, \bC)=H^{p+q} (D, \mc O_D)$.
\end{rema}

We end with the following lemma.

\begin{lem} \label{cuproduct}
The spectral sequence of Corollary \ref{spectral sequence}, for $\mathcal{O}_{X_0}$, is endowed with an algebra structure that is compatible with the cup product on $H^*(\mathcal O_{X_0})$.
\end{lem}
\begin{proof}
By Proposition \ref{resolution}, it is enough to produce a morphism of complexes
\be \label{cup product}
\varphi: \mathcal O_{X_0^ \bullet} \otimes \mathcal O_{X_0^\bullet} \to \mathcal O_{X_0^ \bullet}
\ee
which induces the regular cup product on $\mc O_{X_0}$. For $\alpha=\{ \alpha_J\}$  a section of  $\mathcal O_{X_0^{[s]}}$ and $\beta =\{ \beta_K\}$  a section of  $\mathcal O_{X_0^{[t]}}$ we set $\varphi(\alpha \otimes \beta)_{j_0 j_1 \cdots j_{s+t+1}}=\alpha_{j_0 j_1 \cdots i_s}|_{{X_0}_{j_0 j_1 \cdots j_{s+t+1}}} \cdot \beta_{j_s j_{s+1} \cdots i_{s+t+1}}|_{{X_0}_{j_0 j_1 \cdots j_{s+t+1}}}
$.
The verification that $\varphi$ is a morphism of complexes is formally the same as that for the cup product in Cech cohomology.
We can lift $\varphi$ to a morphism of the Cech resolutions of each of the two complexes, getting a product structure on the corresponding spectral sequence and hence a product $H^q(\mc O_{X_0^{[p]}})\otimes H^{q'}(\mc O_{X_0^{[p']}}) \to H^{q+q'}(\mc O_{X_0^{[p+p']}})$ which is compatible with the cup product on $H^*(\mc O_{X_0})$.
\end{proof}

 \bibliography{hkref2}
\end{document}